\documentclass[10pt]{article}

\usepackage{amsmath,amssymb,amsfonts}
\usepackage{amsthm}
\usepackage{mathtools}
\usepackage{nicefrac}
\usepackage{authblk}
\usepackage{verbatim}
\usepackage{mathrsfs}
\usepackage{subfigure}
\usepackage{latexsym}
\usepackage{dsfont}
\usepackage{xcolor}
\usepackage{fancybox}
\usepackage{blindtext}
\usepackage{faktor}
\usepackage{geometry}
\usepackage{todonotes}
\usepackage{tikz-network}
\usepackage{lineno}
\usepackage{graphicx}
\usepackage[ruled,linesnumbered]{algorithm2e}
\usepackage{hyperref}
\hypersetup{colorlinks=true,breaklinks,linkcolor=blue,urlcolor=blue,anchorcolor=clemsonorange,citecolor=black}

\newtheorem{theorem}{Theorem}[section]

\newtheorem{lemma}[theorem]{Lemma}
\newtheorem{proposition}[theorem]{Proposition}
 
\theoremstyle{definition}

\newtheorem{definition}[theorem]{Definition}
\newtheorem{remark}[theorem]{Remark}
\newtheorem{example}[theorem]{Example}

\newcommand{\rr}{\mathbb{R}}

\newcommand{\nn}{\mathbb{N}}

\newcommand{\cP}{\mathcal{P}}
\newcommand{\cQ}{\mathcal{Q}}
\newcommand{\cR}{\mathcal{R}}
\newcommand{\cT}{\mathcal{T}}

\newcommand{\overbar}[1]{\mkern 1.5mu\overline{\mkern-1.5mu#1\mkern-1.5mu}\mkern 1.5mu}

\DeclareMathOperator{\cone}{cone}

\newcommand{\setof}[1]{\left\{#1\right\}}

\newcommand{\paren}[1]{\left( #1 \right)}

\newcommand{\poset}{\mathcal{P}}
\newcommand{\evd}{\Xi} 
\newcommand{\parm}{\xi} 

\newcommand{\aclep}{AC-LEP } 
\newcommand{\ord}{\sigma} 
\newcommand{\PL}{\mathbf{n}}

\newcommand{\totord}{\mathcal{T}} 
\newcommand{\vob}{V_{0}}
\newcommand{\cob}{C_{0}}
\newcommand{\psd}{\left(\poset, \prec_B, \evd\right)}
\newcommand{\psdR}{\left(\poset, \prec_B, (0,\infty)^{2n}\right)}
\newcommand{\lpsdR}{\left(\poset', \prec_B, \rr^{m}\right)}
\newcommand{\lpsd}{\left(\poset', \prec_B, \rr^m \right)}

\newcommand{\RV}{\mathbf{u}}

\UseRawInputEncoding 

\title{Computing linear extensions for polynomial posets subject to algebraic constraints}
\author{Shane Kepley\thanks{{sk2011@math.rutgers.edu}}, Lun Zhang\thanks{lz210@math.rutgers.edu}, and  Konstantin Mischaikow\thanks{mischaik@math.rutgers.edu}}
\date{}
\begin{document}
\maketitle
\begin{abstract}
In this paper we consider the classical problem of computing linear extensions of a given poset which is well known to be a difficult problem. However, in our setting the elements of the poset are multivariate polynomials, and only a small ``admissible'' subset of these linear extensions, determined implicitly by the evaluation map, are of interest. This seemingly novel problem arises in the study of 
global dynamics of gene regulatory networks in which case the poset is a Boolean lattice. We provide an algorithm for solving this problem using linear programming for arbitrary partial orders of linear polynomials. 
 This algorithm exploits this additional algebraic structure inherited from the polynomials to efficiently compute the admissible linear extensions. The biologically relevant problem
involves multilinear polynomials and we provide a construction for embedding it into an instance of the linear problem.
\end{abstract}

\section{Introduction}
\label{sec:introduction}

Consider a set of real polynomials $\poset$, defined on a domain $\Xi \subset \rr^d$, equipped with a partial order $\prec$. We are interested in identifying linear extensions (total orders compatible with $\prec$) that are satisfied by $\poset$ under evaluation at a point in $\Xi$.
To be more precise consider a semi-algebraic set $\Xi\subset \rr^d$, called the {\em evaluation domain}, and a collection of polynomials  $\poset := \setof{p_0,\ldots ,p_K}\subset \rr[x_1,\ldots,x_d]$.
Let $\prec$ denote a partial order on $\poset$ such that if $p \prec q$, then
\begin{equation}
\label{eq:primaryPO}
p(\xi) < q(\xi) \quad \text{for all} \ \xi\in \Xi.
\end{equation}

Let $S_{K+1}$ denote the set of permutations on $K+1$ symbols. We identify linear extensions of $\poset$ with a subset of $S_{K+1}$ as follows. Given $\sigma \in S_{K+1}$, let $\prec_\sigma$ denote the linear order
\[
p_{\sigma(0)} \prec_\sigma p_{\sigma(1)} \prec_\sigma \cdots \prec_\sigma p_{\sigma(K)}.
\]
We define the {\em realizable set} associated to $\sigma$ by

\begin{equation}\label{eq:xisigma}
\evd_\sigma := \setof{\parm \in \evd : p_{\sigma(k)}(\parm) < p_{\sigma(k+1)}(\parm) \ \text{for all} \ 0 \leq k \leq K-1}.
\end{equation}
Observe that if $\Xi_\sigma \neq \emptyset$, then $\prec_\sigma$ is a linear extension of $\prec$.
The \emph{algebraically constrained linear extension problem} (AC-LEP) defined by $(\poset,\prec,\evd)$ is to determine
\begin{align*}
\cT (\poset,\prec,\evd) :=  \setof{\ord\in S_{K+1} :  \evd_\sigma \ \text{is nonempty}}.
\end{align*}
Notice that in the formulation of the AC-LEP we have identified the partial order $\prec_\sigma$ with the associated element in $S_{K+1}$.
We will use this identification throughout the remainder of the paper.
We say that each $\sigma \in \cT (\poset,\prec,\evd)$ is an {\em admissible} linear extension.

As is discussed in Section~\ref{sec:dsgrn} our immediate motivation for studying the AC-LEP comes from modeling the dynamics of regulatory networks in biology and in particular characterizing relevant subsets of the parameter space. For the moment we attempt to put this problem into a broader mathematical context as the problem itself, as well as our solutions for some special cases, have elements of both classical real algebraic geometry and order theory. 

\subsubsection*{Quantifier elimination and real algebraic geometry}
Observe that if $\Xi=\rr^d$, and $\poset$ is an arbitrary collection of polynomials, then $\sigma \in S_{K+1}$ is admissible if and only if there exists $\xi\in \rr^d$ such that $p_{\sigma(k)}(\xi) - p_{\sigma(k+1)}(\xi) <0$ for all $0 \leq k \leq K$. 
These inequalities define a semi-algebraic set.
Therefore, if $\prec$ is the trivial partial order (i.e.\ $\poset$ is a single anti-chain), then this instance of AC-LEP is equivalent to the classical problem of decidability for real semi-algebraic sets. 

The previous example illustrates a major challenge in solving the AC-LEP. The first general algorithm for solving the quantifier elimination/decidability problem for polynomials in $\rr^d$ with feasible running time was the cylindrical algebraic decomposition (CAD) given by Collins \cite{Collins1975} in 1975. The CAD algorithm works by subdividing $\evd$ into subsets on which the polynomials are sign invariant. Given such a decomposition, decidability is reduced to simply evaluating each polynomial at a sample point located in each subset and checking if it satisfies the necessary inequalities. Unfortunately, the computational complexity of the algorithm grows like
\begin{equation}
\label{eq:CAD_complexity}
\mathcal{O} \left( 
\left( 2D \right)^{2^d-1}(K+1)^{2^d -1}2^{2^{d-1}-1}\right) \quad \text{where} \quad D = \max \setof{\deg p : p \in \cP}.
\end{equation}
This worst case running time is known to be sharp even for classes of ``nice'' polynomials e.g. linear \cite{Brown2007}, and moreover, the worst case is also typical \cite{MR1998147}. As a result, the question of whether or not 
$\sigma \in \totord \left(\poset, \prec, \evd \right)$,
even for a single $\sigma \in S_{K+1}$, is often intractable for problems of practical interest. 

In addition, the CAD algorithm does not provide partial information. It either runs to completion, in which case it is guaranteed to provide an answer, or it provides no information. 
Furthermore, we note that if additional algebraic constraints are added e.g.~we assume $\evd_0 \subset\evd \subset \rr^d$ is a strict semi-algebraic subset, then the CAD algorithm can handle this by simply appending the polynomial constraints which define $\evd_0$ to the set of polynomials. 
However, this dramatically increases the complexity of the CAD algorithm, despite the fact that the number of admissible linear extensions can only decrease.   

Some improved algorithms have been proposed which aim to reduce the complexity of specific aspects of the problem or for special classes of polynomials (e.g.~\cite{McCallum1985, Brown2001, Brown2013}). 
These improvements often provide dramatic algorithmic speedups for checking whether $\sigma \in \totord \left(\poset, \prec, \evd \right)$ for a single linear extension. However, these algorithms have the same worse case running time as the general CAD algorithm and understanding which classes of polynomials benefit is still a very active area of research. 
Therefore, these improved algorithms alone are not sufficient to handle instances of AC-LEP since we are interested in determining which of the $(K+1)!$ possible semi-algebraic sets are nonempty. 
An efficient algorithm that does not produce a decomposition of $\evd$ into sign invariant subsets would still need to be called $(K+1)!$ times. 
However, we will make use of these improved algorithms as a post-processing step which we discuss further in Section~\ref{sec:PSD}. 

\subsubsection*{Computing linear extensions of Boolean lattices}
Let us momentarily ignore the algebraic structure in the AC-LEP  by forgetting that $\poset$ is a collection of polynomials. Hence, we focus only on the poset structure $(\poset, \prec)$, and consider the problem of computing all linear extensions. 
The related problem of counting all linear extensions of a partial order is a well studied problem in order theory. Its importance is due in large part to its connection with the complexity of sorting elements in a list. If one considers a list of $(K+1)$ distinct values which have been partially sorted by making pairwise comparisons on a subset of its elements, then these comparisons induce a partial order. Therefore, the linearly ordered values of the fully sorted list are given by one of the possible linear extensions for the partial order. As a result, the complexity of completely sorting a list is intimately connected to counting linear extensions for posets.

Observe that computing the set of all linear extensions of $(\poset, \prec)$ is not easier than counting them which is known to be $\#P$-complete \cite{Brightwell:1991:CLE:103418.103441}. In particular, a polynomial time algorithm for computing all possible linear extensions for arbitrary posets would imply that $P = NP$ by Toda's theorem \cite{Toda1991}. Moreover, we are interested not only in counting linear extensions, but explicitly computing them. Therefore, we are also concerned with how fast the number of admissible linear extensions grows. 

For reasons we discuss in Section \ref{sec:PSD}, we are specifically interested in the case that $(\poset, \prec)$ is a Boolean lattice. 
Specifically, for fixed $n \in \nn$, define $S_n := \setof{1,\dots,n}$ and let $2^{S_n}$ denote its power set. The {\em standard} $n$-dimensional Boolean lattice is the poset, $(2^{S_n}, \prec_B)$, where $\prec_B$ is the partial order defined by inclusion. We say a poset, $(\poset, \prec)$, is an $n$-dimensional Boolean lattice if $(\poset, \prec)$ is order isomorphic to the standard $n$-dimensional Boolean lattice and we write $\prec_B$ in place of $\prec$. 

Estimating the number of linear extensions for Boolean lattices was first considered in  \cite{Sha1987} which established a nontrivial upper bound. Later, Brightwell and Tetali \cite{Brightwell2003} proved the following result that essentially settles the question for all practical considerations. 
If $Q(n)$ denotes the number of linear extensions of an $n$-dimensional Boolean lattice, then
\begin{equation}
\label{eq:BL_linear_extension_estimate}
\frac{\log Q(n)}{2^n} = \log 
\binom{n}{ \lfloor \nicefrac{n}{2} \rfloor} - \frac{3}{2}\log e + \mathcal{O} \left( \frac{\ln n}{n} \right).
\end{equation}

The estimate in Equation \eqref{eq:BL_linear_extension_estimate} illustrates a major challenge in solving the instances of AC-LEP of interest in this paper. Consider an instance of AC-LEP given by $\psd$ where $(\poset, \prec_B)$ is an $n$-dimensional Boolean lattice and $\evd$ is any evaluation domain. 
Suppose we had a black box for efficiently computing all linear extensions of a Boolean lattice denoted by $L \subset S_{K+1}$. 
Furthermore, assume we also had a ``CAD''-like algorithm which could efficiently check if $\Xi_\ord \neq \emptyset$. 
Then, one would need to call this algorithm only $\# L$-many times as opposed to $(K+1)!$ as we argued above. 
However, the growth estimate in Equation \eqref{eq:BL_linear_extension_estimate} implies that the number of calls to this algorithm would still grow exponentially. 

\subsubsection*{This work}
In this work we present efficient algorithms for solving two specific instances of the AC-LEP. The first is the \emph{linearly constrained linear extension problem} (LC-LEP), in which $\poset$ is a set of linear polynomials, $\evd$ is the interior of a polyhedral cone (i.e.~$\evd$ is defined by a finite number of linear inequalities), and $\prec$ is an arbitrary partial order. 
We present an efficient algorithm for solving the LC-LEP in Section \ref{sec:LCLEP}. 
The second instance of the AC-LEP, which we call the {\em parameter space decomposition} (PSD) problem and describe now
(see Definition~\ref{defn:PSDProblem} for a precise definition),
is motivated by an application from systems biology  described in Section~\ref{sec:dsgrn}. 
\begin{definition}
\label{def:interaction_function}
	For $n \in \nn$, an \emph{interaction function} of order $n$ is a polynomial in $n$ variables, $z = 
	\left(z_1,\dotsc, z_n \right)$, of the form
	\begin{equation}
	\label{eq:interaction_function}
	f(z) = \prod_{j=1}^{q} f_j(z) 
	\end{equation}
	where each factor has the form 
	\[
	f_j(z) = \sum_{i \in I_j} z_i 
	\]
	and the indexing sets $\setof{I_j: 1 \leq j \leq q}$ form a partition for $\setof{1,\dots,n}$. We define the \emph{interaction type} of $f$ to be $\PL := \left(n_1,\dotsc,n_q\right)$ where $n_j$ denotes the number of elements in $I_j$.
\end{definition}

\begin{remark}
\label{rem:ordering}
We leave it to the reader to check that the order of the indexing sets $I_j$ does not matter for any of the analysis in this paper.
Therefore, for convenience in reporting results (see Section~\ref{sec:results}) we will always assume that 
\[
n_1\geq n_2\geq \cdots \geq n_q.
\]
\end{remark}

To define an instance of the PSD problem, fix an interaction function $f$ of order $n$ and let $\poset$ be the collection of polynomials in the $2n$ positive real variables, $\setof{\ell_i, \delta_i : 1 \leq i \leq n}$, obtained by evaluating $f(z)$ with each $z_i \in \setof{\ell_i,\ell_i+\delta_i}$. 
Taking all possible combinations of $z_i$ for $1 \leq i \leq n$ produces the polynomials for the PSD problem, 
\begin{equation}
    \label{eq:polynomials}
\cP = \setof{p_0,\ldots,p_{2^n-1}} \subset \rr[\ell_1,\ldots, \ell_n,\delta_1,\ldots,\delta_n].
\end{equation}
In Section \ref{sec:PSD}, we will define an indexing map between the $2^n$ elements of $\poset$ and  the standard $n$-dimensional Boolean lattice. 
 Let $\prec_B$ denote the Boolean lattice partial order with respect to this index map, and set $\Xi = (0,\infty)^{2n}$. 
The \emph{PSD problem} is the instance of the \aclep defined by $\psd$. In Section \ref{sec:PSD}, we prove that $(\poset, \prec_B, \evd)$ satisfies Equation \eqref{eq:primaryPO}. However, we present some examples before continuing. 
\begin{example}
\label{ex:3input}
The simplest nonlinear PSD problem arises from the interaction function
\[
f(z) = (z_1+z_2)z_3
\]
which has interaction type, $\PL = (2,1)$. The PSD polynomials for this interaction function are given by
\begin{align*}
    p_0 & = (\ell_1 + \ell_2)\ell_3 & p_4 & = (\ell_1 + \ell_2 + \delta_1)\ell_3 \\
p_1 & = (\ell_1 + \ell_2)(\ell_3 + \delta_3) & p_5 & = (\ell_1 + \ell_2 + \delta_1 )(\ell_3 + \delta_3) \label{eq:pi3node}\\
p_2 & = (\ell_1 + \ell_2 + \delta_2)\ell_3 & p_6 & = (\ell_1 + \ell_2 + \delta_1 + \delta_2)\ell_3 \\
p_3 & = (\ell_1 + \ell_2 + \delta_2)(\ell_3 + \delta_3) &	p_7 & = (\ell_1 + \ell_2 + \delta_1 + \delta_2)(\ell_3+ \delta_3).
\end{align*}
The PSD evaluation domain is $\Xi = (0,\infty)^6$ and the partial order, $\prec_B$, is imposed on $\poset$ by identifying $p_i$ with the vertex of a unit cube whose coordinates are $(i)_2 \in \mathbb{F}_2^3$ where $(i)_2$ is the binary expansion of $i$. The solution to this PSD problem is the set of admissible linear extensions of $(\cP, \prec_B)$, such that $\ord \in \totord \left(\poset, \prec_B, (0,\infty)^{6} \right)$ if and only if $\evd_\ord \neq \emptyset$. We note that there are $8! = 40,320$ linear orders on $\poset$. However, only $48$ of these are linear extensions of $\prec_B$, and of these, only the following $20$ linear extensions are admissible. 
\begin{equation*}
\begin{aligned}
(0, 1, 2, 3, 4, 5, 6, 7) \\ 
(0, 1, 2, 3, 4, 6, 5, 7) \\
(0, 1, 2, 4, 3, 5, 6, 7) \\
(0, 1, 2, 4, 3, 6, 5, 7) \\
(0, 4, 2, 6, 1, 5, 3, 7) 
\end{aligned}
\quad 
\begin{aligned}
 (0, 1, 2, 4, 6, 3, 5, 7) \\
 (0, 1, 4, 2, 5, 3, 6, 7) \\
 (0, 1, 4, 2, 5, 6, 3, 7) \\
 (0, 1, 4, 2, 6, 5, 3, 7) \\
 (0, 4, 1, 5, 2, 6, 3, 7)
\end{aligned}
\quad 
\begin{aligned}
 (0, 1, 4, 5, 2, 3, 6, 7) \\
 (0, 1, 4, 5, 2, 6, 3, 7) \\
 (0, 2, 1, 3, 4, 6, 5, 7) \\
 (0, 2, 1, 4, 3, 6, 5, 7) \\
 (0, 4, 2, 1, 6, 5, 3, 7)
\end{aligned}
\quad 
\begin{aligned}
 (0, 2, 1, 4, 6, 3, 5, 7) \\
 (0, 2, 4, 1, 6, 3, 5, 7) \\
 (0, 2, 4, 6, 1, 3, 5, 7) \\
 (0, 4, 1, 2, 5, 6, 3, 7) \\
  (0, 4, 1, 2, 6, 5, 3, 7)
\end{aligned}
 \end{equation*}

The $28$ ``missing'' linear extensions are those which do not satisfy certain algebraic constraints which are imposed by the polynomial structure. 
For example, observe that for fixed $\parm \in (0, \infty)^6$, if $p_3(\parm) < p_6(\parm)$, then $p_1(\parm) < p_4(\parm)$ must also hold.

Unlike the partial order which constrains all possible linear extensions, this order relation is conditional. 
Indeed, there exist choices of $\parm$ such that $p_3(\parm) > p_6(\parm)$ in which case there is no requirement imposed on the order of $p_1(\parm), p_4(\parm)$, and in fact, there are admissible linear extensions which satisfy both choices, e.g.~the first two orders in column four. 
As another example, observe that $p_5(\parm) < p_6(\parm)$, if and only if $p_1(\parm) < p_2(\parm)$. 
This algebraic constraint is bi-conditional.
However, it too can not be represented in the partial order since both choices occur in at least one admissible order. 
\end{example}
To emphasize the role of the interaction function in determining the algebraic constraints, we consider a similar PSD problem that is also an instance of LC-LEP.
\begin{example}
\label{ex:3inputLin}
Consider the interaction type, $\mathbf{n} = (3)$ with corresponding interaction function
\[
f = z_1+z_2 + z_3.
\]
As in Example \ref{ex:3input}, we obtain $8$ PSD polynomials given explicitly by
\begin{align*}
		p_0 & = \ell_1 + \ell_2 +\ell_3 & p_4 & = \ell_1 + \ell_2  +\ell_3 + \delta_1 \\
p_1 & = \ell_1 + \ell_2+\ell_3 + \delta_3 & p_5 & = \ell_1 + \ell_2  + \ell_3 + \delta_1 + \delta_3 \\
p_2 & = \ell_1 + \ell_2 + \ell_3 + \delta_2 & p_6 & = \ell_1 + \ell_2  +\ell_3 + \delta_1 + \delta_2 \\
p_3 & = \ell_1 + \ell_2 +\ell_3 + \delta_2  + \delta_3 &	p_7 & = \ell_1 + \ell_2  +\ell_3+ \delta_1 + \delta_2 + \delta_3
\end{align*}
The evaluation domain and partial order are identical to the PSD problem in Example \ref{ex:3input}. Nevertheless, only the following $12$ linear extensions are admissible 
\begin{equation*}
\begin{aligned}
(0, 1, 2, 3, 4, 5, 6, 7)\\
 (0, 1, 2, 4, 3, 5, 6, 7)\\
 (0, 1, 4, 2, 5, 3, 6, 7)
\end{aligned}
\quad 
\begin{aligned}
(0, 1, 4, 5, 2, 3, 6, 7)\\
 (0, 2, 1, 3, 4, 6, 5, 7)\\
 (0, 2, 1, 4, 3, 6, 5, 7)
\end{aligned}
\quad 
\begin{aligned}
(0, 2, 4, 1, 6, 3, 5, 7)\\
 (0, 2, 4, 6, 1, 3, 5, 7)\\
 (0, 4, 1, 2, 5, 6, 3, 7)
\end{aligned}
\quad 
\begin{aligned}
(0, 4, 1, 5, 2, 6, 3, 7)\\
 (0, 4, 2, 1, 6, 5, 3, 7)\\
 (0, 4, 2, 6, 1, 5, 3, 7)
\end{aligned}
 \end{equation*}
\end{example}

Similarly, the missing $36$ linear extensions in this example fail to satisfy some algebraic constraints. In both cases, the set of admissible linear extensions is a fraction of the set of all linear extensions of the Boolean lattice. In other words, the algebraic structure implies that the admissible linear extensions are a sparse subset of all linear extensions. The algorithm in this paper exploits the algebraic and order theoretic aspects of the PSD problem to overcome the computational complexity limitations which plague both problems in general.
Furthermore, we prove that this algorithm finds all possible linear extensions.
For both examples we obtained the ($12$ and $20$ respectively) admissible solutions {\em without} first computing the linear extensions of $(\poset, \prec_B)$ and then checking which are admissible.

\subsubsection*{Related work}

As is indicated above our original motivation for this paper arises from problems in systems biology for which explicit complete solutions to the PSD problem are required.
As such the majority of this introduction has focused on the question of efficacy of computation.
However, there is another direction in which the work of this paper overlaps with other efforts.
In particular, observe that the case where the interaction function is linear, i.e. has interaction type $\mathbf{n} = (n)$, solving the AC-LEP is equivalent to identifying all the cells of a hyperplane arrangement.
This latter problem has been the subject of considerable study (see \cite{stanley} for an introduction) and in particular, Maclagan \cite{boolorder} provides the number of solutions for the linear PSD problem for $n = 1,\ldots, 7$.
Our computations (see Table~\ref{table:results}) lead to the same numbers, 
as expected.

After accounting for symmetry in the number of linear PSD solutions for interaction types $(1,1,1,1)$, $(1,1,1,1,1)$, and $(1,1,1,1,1,1)$,  reported in column two of Table~\ref{table:results}, we obtain
\[
\frac{336}{4!} = 14 =: a_4, \quad \frac{61920}{5!} = 516 =: a_5, \quad \frac{89414640}{6!} = 124,187 =: a_6
\]
which align with sequence A009997 in the OEIS \cite{oeis}.
From \cite{comparative}, we know this sequence represents the number of comparative probability orderings on all subsets of $n$ elements that can arise by assigning a probability distribution to the individual elements. 
The equivalence of comparative probability orderings and solutions to the linear PSD problem follows directly from the definition of comparative probability.

\subsubsection*{Organization of paper}
The remainder of this paper is organized as follows. In Section \ref{sec:dsgrn} we briefly describe how the PSD problem arises naturally in the study of global dynamics for gene regulatory networks. 
In Section \ref{sec:LCLEP}, we present an efficient algorithm for solving instances of the LC-LEP. In Section \ref{sec:PSD}, we show that the LC-LEP is related to the PSD problem in the following way. 
If $\left( \poset, \prec_B, (0, \infty)^{2n} \right)$ is an instance of the PSD problem, then we construct an associated instance of LC-LEP, denoted by  $\left(\poset', \prec_B, \evd'\right)$, which satisfies the inclusion 
\begin{equation}
\label{eq:PSD_solution_inclusion}
    \totord \left( \poset, \prec_B, (0, \infty)^{2n} \right) \subseteq \totord \left(\poset', \prec_B, \evd'\right).
\end{equation}
We refer to this instance of LC-LEP as the {\em linearized} PSD problem associated to $\left( \poset, \prec_B, (0, \infty)^{2n} \right)$. We exploit this construction and the algorithm for solving the LC-LEP presented in Section \ref{sec:LCLEP}, to provide a means of efficiently computing a collection of candidates that contains the solution to the PSD problem. We prove that in some cases the inclusion in Equation \eqref{eq:PSD_solution_inclusion} is actually an equality. More generally, this inclusion is strict, but the candidate set is a sparse subset of the collection of all linear extensions of $(\poset, \prec_B)$. In this case we describe algorithms for removing the non-admissible solutions. 

Finally, in Section \ref{sec:results} we present the results for all PSD solutions with order up to four. Additionally, we have some results for PSDs of order five and six. 
For the remaining cases and PSDs of higher order the computations become too large, i.e.\ we do not have immediate use for them in applications and they require significant computing resources.

\section{Dynamic Signatures Generated by Regulatory Networks}
\label{sec:dsgrn} 

This section provides a brief description of how the AC-LEP arises in the context of mathematical modeling of problems from systems biology with a few comments at the end indicating its potential application in more general settings of nonlinear dynamics.
Biologists often describe regulatory networks in terms of annotated directed graphs, such as that shown in Figure~\ref{fig:p53}(a) where the labeling of the edges, $n\to m$ or $n\dashv m$ indicates whether node $n$ activates or represses node $m$.
Our goal is to describe the type of dynamics that can be expressed by the regulatory network. 
This requires two things: (i) a mathematical framework in which we can rigorously discuss dynamics when an analytic expression of the nonlinearity is unknown, and more specifically, (ii) imposing  a mathematical interpretation on the regulatory network that is compatible with its use as a biological model. 

We discuss both of these topics in the following sections, but hasten to add that there are other complementary approaches to this problem.
Perhaps the strongest dichotomy is between  Boolean models and ordinary differential equation (ODE) models (see \cite{albert:collins:glass} for a brief overview and references).
Recent efforts in the ODE direction seek to go from networks to explicit polynomial vector fields (see \cite{JARRAH2007477} and references therein).
The approach we describe lies somewhere between these two.
Our approach is partially combinatorial, and thus we obtain efficacy similar to that of the Boolean models, but
we maintain continuous parameterization (hence the necessity of the results of this paper),
and therefore our computations can be interpreted in the context of ODEs.
However, unlike methods based on explicit vector fields, we derive the structure of the dynamics via order theory and algebraic topology.
Understanding the relative strengths and weaknesses of these various methods is still an active area of research. 

\begin{figure}[h!]
\begin{picture}(400,150)
\put(0,0){(a)}
\put(0,10){
			\begin{tikzpicture}[main node/.style={circle,fill=white!20,draw,font=\sffamily\normalsize\bfseries},scale=2.5]
				\node[main node] (Wip1) at (0,0) {5};
				\node[main node] (p53) at (1,0) {1};
				\node[main node] (Mdm2) at (2,0) {4};
				\node[main node] (Chk2) at (0.5,0.5){2};
				\node[main node] (ATM) at (0.5,1.5) {3};

				\path[->,>=angle 90,thick]
				(Wip1) edge[-|,shorten <= 2pt, shorten >= 2pt] node[] {} (Chk2)
				(Wip1) edge[-|,shorten <= 2pt, shorten >= 2pt] node[] {} (ATM)
				(p53) edge[->,shorten <= 2pt, shorten >= 2pt] node[] {} (Wip1)
				(p53) edge[->,shorten <= 2pt, shorten >= 2pt, bend left] node[] {} (Mdm2)
				(Chk2) edge[->,shorten <= 2pt, shorten >= 2pt] node[] {} (p53)
				(ATM) edge[->,shorten <= 2pt, shorten >= 2pt] node[] {} (p53)
				(ATM) edge[->,shorten <= 2pt, shorten >= 2pt] node[] {} (Chk2)
				(Mdm2) edge[-|,shorten <= 2pt, shorten >= 2pt, bend left] node[]  {} (p53)
				;

			\end{tikzpicture}
		}
			
\put(200,0){(b)}
\put(200,-10){\begin{tikzpicture}

		\draw [->] (0,0) -- (4,0);
		\node at (4,-0.4) {$x_n$};
		\draw [dashed]  (2,0) -- (2,3);
		
		\draw [->] (0,0) -- (0,3);
		\draw  (-0.1,0.75) -- (0.1,0.75);
		\draw  (-0.1,2.25) -- (0.1,2.25);
		\draw [fill] (0.45, 0.45) circle [radius=0.075];
		\draw [fill] (0.6, 0.9) circle [radius=0.075];
		\draw [fill] (0.65, 0.65) circle [radius=0.075];
		
		\draw [fill] (3.4, 2.65) circle [radius=0.075];
		\draw [fill] (3.7, 2.45) circle [radius=0.075];
		\draw [fill] (3.55, 2.15) circle [radius=0.075];
		
		\node at  (-0.85, 0.75)  {$\ell_{m,n}$};
		\draw [dashed]  (0,0.75) -- (2,0.75);
		\node at  (-0.95, 2.25)  {$\ell_{m,n} +\delta_{m,n}$};
		\draw [dashed]  (2,2.25) -- (4,2.25);
		\node at  (2, -0.5)  {$\theta_{m,n}$};
		
		;	
		\end{tikzpicture}
		}						

\end{picture}

\caption{(a) Example of a regulatory network. (b) Model for edge $n\to m$ where $\ell_{m,n}$ indicates low level of growth rate of $x_m$ induced by $x_n$ and $\ell_{m,n}+\delta_{m,n}$ indicates high level of growth rate of $x_m$ induced by $x_n$. $\theta_{m,n}$ provides information about the value of $x_n$ that lies between low values inducing low and high expression levels.}
\label{fig:p53}
\end{figure}
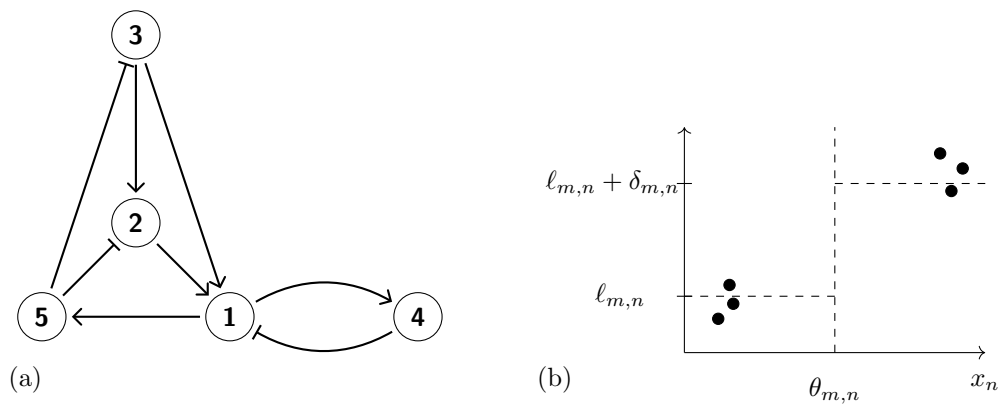

\subsection{Mathematical framework}
\label{sec:mathframework}
Classical nonlinear dynamics focuses on  invariant sets and bifurcations, both of which are extremely sensitive to parameters, e.g.\ nonlinearities. Unfortunately, this is precisely the information that is unknown in multiscale systems such as those associated with biology. 
C. Conley proposed that in these situations one should focus on isolated invariant sets \cite{MR511133} as there is a sheaf structure associated to them \cite{MR334173} and thus are not subject to the above mentioned sensitivity. 
Furthermore, he developed an algebraic topological invariant for isolated invariant sets, called the Conley index, from which one can recover considerable information about the existence and structure of the invariant set \cite{MR1901060}.
In particular, using the homology Conley index one can identify the existence of invariant sets \cite{MR511133}, equilibria \cite{srzednicki85,mccord88}, periodic orbits \cite{mccord:mischaikow:mrozek}, heteroclinic orbits \cite{conley:smoller},  chaotic dynamics \cite{mischaikow:mrozek:95,szymczak96}, and semi-conjugacies onto nontrivial dynamics \cite{mischaikow95,MR1354959}.

One indication of the strength of Conley's proposition is that it is possible to study a differential equation (ordinary or partial) or a continuous map (finite or infinite dimensional) numerically and with appropriate bounds on the errors, to draw rigorous conclusions about the dynamics from homological Conley index computations \cite{MR1276767, MR2136516, MR2028588, MR2067140}.
The weakness of the above mentioned results is that they are perturbative in nature, i.e.\ the results are computed at given parameter values and the results are guaranteed to be true for a sufficiently small neighborhood of that parameter value. In practice one can extract numerical bounds for the size of the neighborhood, but they tend to be small.

For multiscaled systems such as those arising in systems and synthetic biology we need to be able to investigate dynamics over large ranges of parameter values.
This was done in the context of low dimensional maps with a low dimensional parameter space \cite{MR2533624,MR3388721,MR3484393}.
The strategy is to subdivide parameter space uniformly, for each region of parameter space compute dynamics with error bounds  valid over the entire region, and once again interpret the structure of the dynamics using the Conley index.
While the approach is extremely effective for low dimensional problems it is  difficult to extend to higher dimensions for two reasons:
\begin{description}
\item[R1] the cost of computing effective numerical error bounds becomes prohibitive, and
\item[R2] the subdivision of parameter space is not chosen according to the underlying dynamics and refinement of the subdivision leads to exponential growth with dimension in the number of subdivisions.
\end{description}

With {\bf R1} in mind W. Kalies, R. Vandervorst, and the third author have developed an approach to dynamics that is based on combinatorics, order theory (posets and lattices), and algebraic topology.
The combinatorics arises from the choice of a decomposition of phase space into closed regular sets and the dynamics is characterized by a relation on this collection of closed regular sets.
It is proven that Conley's fundamental characterization of dynamics in terms of attractor-repeller pairs or equvalently Morse decompositions can be completely recovered via this combinatorial approach \cite{MR2189545,MR3415257,MR3552843}.
The algebra of the order theory, in particular Birkhoff's theorem relating finite posets with finite distributive lattices, provides an algorithmic approach for the identification of index pairs from which the Conley index can be computed \cite{2019arXiv191109382K}.
This in turn allows one to obtain a chain complex that codifies  gradient-like structure of the dynamics and the Conley indices of all the associated isolated invariant sets \cite{2018arXiv181004552H}.
Observe that, at least conceptually, we have replaced the numerical computations by combinatorial computations.
In practice, at least for the biological regulatory networks discussed below, this combinatorial approach is extremely efficient, both with respect to time and memory \cite{cummins:gedeon:harker:mischaikow:mok}.

Dealing with {\bf R2} is the raison d'{e}tre for this paper.
As is made clear below for the regulatory networks of interest our closed regular sets take the form of cubes determined by parameter values.
The relation on these cubes that represents the dynamics is obtained by determining the directions of inequalities.
All possible sets of directions are in turn determined by understanding the orders of the polynomials $\cP$ defined by \eqref{eq:polynomials}.
Thus, solving the \aclep provides us with an a priori optimal potentially nonlinear decomposition of parameter space.

\subsection{Mathematical interpretation of regulatory networks}
\label{sec:interpretation}
We now turn to our  mathematical interpretation of a regulatory network that is compatible with its use as a biological model. 
With this in mind, we assign to node $m$ a state variable, $x_m>0$, that corresponds to the quantity of a protein, mRNA, or a signaling molecule. 
Precise nonlinear expressions for the interactions of the variables are \emph{not} assumed to be known, but we do assume that the sign of the rate of change of $x_m$ is determined by the expression
\begin{equation}
\label{eq:basic}
    -\gamma_m x_m + \Lambda_m(x)
\end{equation} 
where $\gamma_m$ indicates the decay rate and $\Lambda_m$ is a parameter dependent function that characterizes the  rate of growth of $x_m$.
Note that $\Lambda_m$ is a function of $x_n$ if and only if there exists an edge from $n$ to $m$ in the regulatory network.

Since the biological model provides minimal information about the effect of $x_n$ on $x_m$ we want to choose a mathematical expression with a minimal set of assumptions.
The rates of growth of $x_m$ due to $x_n$ are labeled as $0 < \ell_{m,n} < \ell_{m,n}+\delta_{m,n}$.
Figure~\ref{fig:p53}(b)  corresponds to an edge $n\to m$ and $\theta_{m,n}$ indicates that the rate of increase $\ell_{m,n}$ must occur at some lesser value of $x_n$ and the rate of increase $\ell_{m,n}+\delta_{m,n}$ must occur at some greater value of $x_n$. An arrow of the form $n\dashv m$ leads to the opposite relation.

This introduces three positive parameters, $\ell_{m,n}$, $\delta_{m,n}$, and $\theta_{m,n}$, for each edge in the regulatory network. 
Note that this is the minimal number of parameters that allows one to quantify the assumption that $x_n$ activates $x_m$ (or equivalently that $x_n$ represses $x_m$). 
We encode this information with the following functions
\begin{equation}
    \label{eq:lambdaplusminus}
\lambda^+_{m,n}(x_n) := \begin{cases}
\ell_{m,n} & \text{if $x_n < \theta_{m,n}$} \\
\ell_{m,n} + \delta_{m,n}& \text{if $x_n > \theta_{m,n}$}
\end{cases}
\quad\text{and}\quad
\lambda^-_{m,n}(x_n) := \begin{cases}
\ell_{m,n} + \delta_{m,n} & \text{if $x_n < \theta_{m,n}$} \\
\ell_{m,n}& \text{if $x_n > \theta_{m,n}$.}
\end{cases}
\end{equation}
We do \emph{not} assume that the values of $\ell_{m,n}, \delta_{m,n}$, or $\theta_{m,n}$ are known (see Figure~\ref{fig:p53}(b)). 
This is intentional as many of these parameters do not have an easy biological interpretation and/or correspond to physical constants which are difficult or impossible to precisely measure. 
Thus, the goal is not to determine the dynamics at any particular choice of parameters, but to determine the range and robustness of the qualitative dynamics exhibited by a network.

A regulatory network such as that of Figure~\ref{fig:p53}(a) does not indicate how multiple inputs to a particular node should be processed. 
An approach that is used is to assume a simple algebraic relationship involving sums and products of the $\lambda^\pm$.  
As an example, a reasonable choice for $x_1$ of Figure~\ref{fig:p53}(a) is \begin{equation}
    \label{eq:3nodeLambda}
\Lambda_1(x_2,x_3,x_4) = \left(\lambda^+(x_2) + \lambda^+(x_3) \right)\lambda^-(x_4).
\end{equation}
Observe that each $\lambda^{\pm}$ takes only two values and therefore, generically Equation \eqref{eq:3nodeLambda} takes $8$ distinct values which are precisely the values of the elements of $\poset$ in the PSD of Example~\ref{ex:3input}.

As is suggested in the caption of Figure~\ref{fig:p53}(b), we do not interpret the values of $\lambda^\pm$, or $\Lambda$ as literal expressions of the nonlinear interactions, but rather, that the associated parameter values are landmarks of whatever of the ``true'' nonlinear function is. This has several consequences.
\begin{enumerate}
    \item We cannot expect Equation \eqref{eq:basic} to provide precise information about the growth rate of $x_m$. 
    Therefore we restrict our attention to asking whether  $x_m$ is increasing or decreasing.
    However, we wish to answer this question over all the possible parameter values $\gamma$, $\theta$, $\ell$ and $\delta$.
    \item The only values of $x_m$ at which the dynamics of $x_n$ change are of the form $\theta_{m,*}$. The associated hyperplanes $x_j = \theta_{k,j}$ decompose phase space $[0,\infty)^{N}$, where $N$ is the number of nodes in the network, into $N$-dimensional rectangular regions called \emph{domains}.
    \item Since we have determined $\cT\left(\setof{p_0,\ldots, p_7},\prec,(0,\infty)^6\right)$ for \eqref{eq:3nodeLambda} we can determine all possible signs of \eqref{eq:basic} associated with \eqref{eq:3nodeLambda} by cataloguing the relative values of $\gamma_1\theta_{j,1}$, $j=4,5$ with respect to $\setof{p_0,\ldots,p_7}$.
\end{enumerate}

The Dynamic Signatures Generated by Regulatory Networks (DSGRN) library contains software that, given the information of the form provided by consequence 3, is capable of efficiently building a database of the global dynamics of a regulatory network over all of parameter space \cite{cummins:gedeon:harker:mischaikow:mok, DSGRN}. 
In \cite{cummins:gedeon:harker:mischaikow:mok} the authors considered networks whose nodes have at most three in-edges, and at most three out-edges. This constraint was due to the difficulty in solving the PSD problem and the results of this paper provide a means to vastly expand the capabilities of DSGRN \cite{DSGRN}.
In particular, DSGRN can now handle the algebraic combinations of 4 to 6 in-edges as is indicated in Section~\ref{sec:results}, and arbitrarily many out-edges.

We remark that DSGRN is proving to be a versatile tool in the realm of systems and synthetic biology with applications in the realm of regulatory networks relevant to oncology \cite{cummins:gedeon:harker:mischaikow:mok, gedeon:cummins:harker:mischaikow:plos, xin:cummins:gedeon}, developmental biology \cite{diegmiller:zhang:gameiro:barr:alsous:schedl:shvartsman:mischaikow}, yeast cell cycle \cite{cummins:gedeon:harker:mischaikow}, and synthetic biology \cite{gameiro:gedeon:kepley:mischaikow}.

\subsection{Extensions}
The focus of this paper is on the PSD problem because of the immediate need, arising from applications, to expand the capabilities of DSGRN to handle regulatory networks with nodes that have more than three in and out edges. 
However, as should be clear from Section~\ref{sec:mathframework} the mathematical foundations for our approach is quite general. 
In particular, the dynamics associated with functions described in Section~\ref{sec:interpretation} or more generally with the PSD problem (See Definition~\ref{defn:PSDProblem}) represent a rather special subclass.
Two immediate generalizations that are being pursued in other work, but rely on application of the techniques of this work are as follows.

The first generalization is to replace the constant $\gamma_m$ in \eqref{eq:basic} by a function $\gamma_m(x)$. 
This allows the DSGRN strategy to be applied in the systems biology setting to regulatory networks that involve post-transcriptional regulation. The second generalization is to replace the functions $\lambda^\pm_{m,n}$ given in \eqref{eq:lambdaplusminus} by a step function involving multiple steps without necessarily assuming monotonicity.
In this setting we no longer have the Boolean lattice as the minimal partial order, but the essential arguments of this paper still apply.

\section{Solving the LC-LEP}
\label{sec:LCLEP}

In this Section, we provide an efficient algorithm to solve the LC-LEP defined in Section \ref{sec:introduction}. Note that if $q \in \rr[x_1,\dotsc, x_d]$ is a linear polynomial, then  evaluation of $q$ defines a linear functional on $\rr^d$. Thus, there exists a unique vector $\RV_q \in \rr^d$, that we call the \emph{representation vector} for $q$, satisfying
\[
q(\xi) = \RV_q \cdot \xi \qquad \text{for all} \quad \xi \in \rr^d.
\]
Recall that the evaluation domain for the LC-LEP is the interior of a polyhedral cone. Thus, there exists a collection of linear polynomials $\cQ_\Xi$, such that
\begin{equation}
\label{eq:polytope}
\Xi = \setof{\xi \in \rr^d : \xi \cdot \RV_q >0\ \text{for all} \quad q\in \cQ_\Xi}.
\end{equation}
We assume that \eqref{eq:polytope} is satisfied for the remainder of this section.

To foreshadow our approach recall that by definition $\ord \in \cT (\cP,\prec,\Xi)$ if and only if $\Xi_{\sigma}\neq \emptyset$. Our approach to determining the latter is to recast it in the language of linear algebra on cones in $\rr^d$. 
From this perspective, the problem is equivalent to rigorously solving a linear programming problem and the efficacy of our algorithm is based on the fact that 
this can be done efficiently. 
With this goal in mind, we begin with a few remarks concerning cones and ordered vector spaces.

\subsection{Cones}
\begin{definition}
A subset $C \subset \rr^d$ is a \emph{cone} if $v \in C$ and $\theta\in [0,\infty)$ implies that $\theta v \in C$.
The cone $C$ is \emph{pointed} if it is closed, convex, and satisfies
\begin{equation}
\label{eq:pointed_condition}
C\cap -C = C \cap \setof{-v : v \in C}　=  0. 
\end{equation}
\end{definition}
Observe that \eqref{eq:pointed_condition} implies that a pointed cone does not contain any lines. 
A vector $v\in \rr^d$ is a \emph{conic combination} of $\{v_1,\dots, v_k\} \subset \rr^d$ if $v = \theta_1 v_1 +\dots + \theta_kv_k$ where  $\theta_1, \dotsc, \theta_k\geq 0$.
Suppose $V = \setof{v_1,\dotsc,v_k} \subset \rr^d$.
The \emph{conic hull} of $V$ is given by
\[
\cone(V) := \setof{ \sum_{i=1}^k\theta_iv_i :  0 \leq \theta_i, \ i=1, \dotsc,k}.
\] 

The following result is left to the reader to check.
\begin{proposition}
\label{prop:cccone}
Given $V = \setof{v_1,\dotsc,v_k} \subset \rr^d$, $\cone(V)$ is the smallest closed convex cone that contains $V$.
\end{proposition}
We make use of the following propositions.
\begin{proposition}
\label{prop:dual_cone}
Suppose $V=\{v_0,\dots,v_{m}\} \subset \rr^d$ is a collection of nonzero vectors such that $\cone(V)$ is a pointed cone. 
Then, there exists some $v' \in \rr^d$ such that $v' \cdot v_i > 0$ for all $0 \leq i \leq m$.
\end{proposition}

\begin{proof}

Observe that $-v_m \notin \cone(\setof{v_0,\dots,v_{m-1}})\subset \cone(V)$ since $\cone(V)$ is pointed. Hence $\setof{-v_m}$ and $\cone(\setof{v_0,\dots,v_{m-1}})$ are disjoint, convex, closed subsets of $\rr^d$.

Therefore, by the hyperplane separation theorem \cite{cvxopt}, there exists $v' \in \rr^d$ such that 
$v' \cdot -v_m < 0$
and $v' \cdot v > 0$ for any $v \in \cone(\setof{v_0,\dotsc,v_{m-1}})$.
\end{proof}

\begin{proposition}
\label{prop:cone_extenstion}
Suppose $V=\setof{v_0,\dots,v_{m}} \subset \rr^d$ is a collection of nonzero vectors such that $\cone(V)$ is a pointed cone. If $-v\notin \cone(V)$, then $\cone(V\cup\{v\})$  is  pointed.
\end{proposition}
\begin{proof}
Suppose that $\cone(V\cup\{v\})$ is not pointed.
Then, there exists $w\neq 0$ such that $w,-w \in \cone(V\cup\{v\})$ or equivalently 
\[
w = \sum_{i=0}^{m}\alpha_iv_i + \alpha v\quad\text{and}\quad
-w = \sum_{i=0}^{m}\beta_iv_i + \beta v
\]
where $\alpha_i,\beta_i,\alpha,\beta$ are all nonnegative. 
Note that if $\alpha=\beta=0$, then $\pm w\in \cone(V)$, which contradicts the assumption that $V$ is pointed.
The sum of the two equations above is 
\[
-(\alpha+\beta) v = \sum_{i=0}^{m}(\alpha_i+\beta_i) v_i.
\]
This implies that $-v\in \cone(V)$, contradicting the assumption that $\cone(V)$ is pointed.
\end{proof}

The previous propositions illustrate the importance of solving the {\em cone inclusion} problem: given a vector $v \in \rr^d$, and finite set of vectors $V \subset \rr^d$, determine whether or not $v\in \cone(V)$. Algorithm 1, stated below, solves this problem. 
Observe that checking if $v\in \cone(V)$ is equivalent to solving the following linear programming feasibility problem.
\begin{align}\label{problem:incone}
\text{Does there exist} \quad \alpha &\nonumber \\
\text{such that} \quad
\textbf{V}\alpha &= v \\
\text{and} \quad\alpha &\geq 0? \nonumber
\end{align}
where $\mathbf{V}$ is the column matrix of $V$. 

Linear programming is a powerful tool that is widely used in convex optimization, and as a result, there are many available solvers/algorithms for solving the linear programming feasibility problem \cite{linprog}. 
The results can be made rigorous by performing computations using interval arithmetic \cite{suprajitno2010linear} or rational linear programming \cite{applegate2007exact} in the case that $\mathbf{V}$ is rational. Observe that the PSD problem defined in Section \ref{sec:introduction} satisfies this constraint. As is made clear in Section~\ref{sec:results}, we use different solvers depending on the machine employed to do the computations. In any case, we take for granted the existence of a rigorous solver for the feasibility problem in Equation \ref{problem:incone} as a ``black box'' which we call {\tt LPSolver} which is employed in the following algorithm.

\begin{algorithm}[H]\label{algorithm:incone}
\SetAlgoLined
\DontPrintSemicolon
\SetKwBlock{Begin}{}{end}
\SetKwComment{Comment}{}{}
\SetKwFunction{FExtension}{InCone}
 \SetKwComment{Comment}{}{}
\SetKwInOut{Results}{Result}
\DontPrintSemicolon
\KwIn{$v, V=\{v_1,\dots,v_m\}, {\tt LPSolver}$}
\KwOut{$\textbf{True},\textbf{False}$}
\Results{Return $\textbf{True}$ if $v\in \cone(V)$ otherwise $\textbf{False}$} 
\caption{Cone inclusion}
    \SetKwFunction{FMain}{InCone}
    \SetKwFunction{LPSolver}{LPSolver}
    \SetKwProg{Fn}{Function}{:}{}
    \Fn{\FMain{$v, V,\LPSolver$}}{
    Return \LPSolver{$v,V$}
    }
    
\textbf{End Function}
\end{algorithm}

The next algorithm uses Proposition~\ref{prop:cone_extenstion} (see line 4) and Algorithm~\ref{algorithm:incone} to determine if a set of vectors defines a pointed cone.

\begin{algorithm}[H]\label{algorithm:pointedness}
\SetAlgoLined
\DontPrintSemicolon
\SetKwBlock{Begin}{}{end}
\SetKwComment{Comment}{}{}
\SetKwFunction{FExtension}{SampleFinder}
 \SetKwComment{Comment}{}{}
\SetKwInOut{Results}{Result}
\DontPrintSemicolon
\KwIn{$V=\{v_1,\dots,v_m\}$}
\KwOut{$\textbf{True},\textbf{False}$}
\Results{Return $\textbf{True}$ if $\cone(V)$ is pointed otherwise $\textbf{False}$} 
\caption{Cone pointedness}
    \SetKwFunction{FMain}{CheckCone}
    \SetKwFunction{InCone}{InCone}
    \SetKwProg{Fn}{Function}{:}{}
    \Fn{\FMain{$V$}}{
    $V' = \{v_1\}$
    
    \For{i = 2 ... m}{
        \eIf{\InCone{$-v_i,V'$}}{
            Return $\textbf{False}$
                }{
            $V' = V'\cup \{v_i\}$
                }
        }
    Return $\textbf{True}$
    }
    
\textbf{End Function}
\end{algorithm}

We now show that the LC-LEP can be reformulated as a problem of identifying whether some specific subsets of vectors generate pointed cones.

\begin{definition}\label{definition:V0}
Given an instance of the LC-LEP, $(\cP, \prec, \evd)$, we define the \emph{base cone} as $\cone(\vob) := \cone(V_\Xi \cup V_\prec)$ where $V_\Xi$ and $V_\prec$ are defined as follows.
Set 
\[
V_\Xi := \setof{\RV_q : q\in \cQ_\Xi},
\]
where $\cQ_\Xi$ are the representation vectors as defined in Equation \eqref{eq:polytope}. Applying Algorithm~\ref{algorithm:pointedness} to $V_\Xi$ (and the fact that we assume $\Xi\neq\emptyset$) shows that $\cone(V_{\Xi})$ is pointed. 
Define
\[
V_\prec := \setof{\RV : \text{$\RV$ is the representing vector of $p-q$ where $q \prec p$ and $p,q \in \cP$}}.
\]
Observe that if $(\cP, \prec)$ satisfies Equation \eqref{eq:primaryPO}, then the representation vector for $p - q$ is an element of $V_{\prec}$ by definition.  Therefore, by Proposition \ref{prop:dual_cone}, $V_\prec$ is pointed. 

\end{definition}

The motivation behind our definition of $V_0$ is that it characterizes the algebraic constraints in the LC-LEP in terms of linear algebra that can be efficiently checked. The next proposition shows that the same idea works for linear extensions. 

Given $\ord\in S_{K+1}$, we define
\begin{equation}
\label{eq:Vord}
V_{\ord} := \vob \cup \setof{\RV_{p_{\ord(i+1)}} - \RV_{p_{\ord(i)}} :  p_{\ord(i)}\in\cP, \ i=0,\dotsc, K-1}.
\end{equation}

\begin{proposition}
\label{prop:order_iff_pointed_cone}
For any $\ord\in S_{K+1}$, $\Xi_{\ord}\neq \emptyset$ if and only if $\cone(V_\ord)$ is a pointed cone.
\end{proposition}

This Proposition is a trivial application of a well known result in convex optimization. If $\overbar{\Xi}$ denotes the closure of $\Xi$, then $\overbar{\Xi}$ is a convex cone. Its {\em dual cone} is defined to be the set
\[
\setof{y \in \rr^d : y \cdot \xi \geq 0 \ \text{for all} \ \xi \in \overbar{\Xi}}
\]
and we observe that this set is nothing more than $\cone(V_\Xi)$. Similarly, for each $\sigma \in S_{K+1}$, $\overbar{\Xi_\ord}$ is again a convex cone as its simply a restriction of $\overbar{\Xi}$ obtained by imposing finitely many additional linear constraints and $\cone(V_{_\ord})$ is its associated dual cone. Consequently, Proposition \ref{prop:order_iff_pointed_cone} follows from the fact that a convex cone is pointed if and only if its dual cone is nonempty. A proof of this result can be found in \cite{book:ewald}. We emphasize that the importance of Proposition~\ref{prop:order_iff_pointed_cone} is the implied equivalence
\[
 \totord(\poset, \prec,\Xi) =\setof{\ord : \Xi_{\ord}\neq \emptyset} = \setof{\ord : \cone(V_\sigma) \ \text{is pointed}}.
\]

\subsection{An algorithm for identifying $ \totord(\poset, \prec,\Xi)$}

In this section we present an algorithm for solving an arbitrary instance of the LC-LEP.

\begin{algorithm}[H]\label{algorithm:core algorithm}
\SetAlgoLined
\DontPrintSemicolon
\SetKwBlock{Begin}{}{end}
\SetKwComment{Comment}{}{}
\SetKwFunction{FExtension}{SampleFinder}
 \SetKwComment{Comment}{}{}
\SetKwInOut{Results}{Result}
\DontPrintSemicolon
\KwIn{$\sigma_{\operatorname{part}} = [\text{ }],  \poset,V= V_0, {\tt Ret} = \{\}$}
\KwOut{$\totord(\poset,\prec,\Xi)$}
\Results{{\tt Ret}: collection of all linearly realizable total order under restriction of $V$} 
\caption{LC-LEP solver}

    \SetKwFunction{FMain}{OrderingGenerator}
    \SetKwFunction{FCheckCone}{CheckCone}
    \SetKwFunction{InCone}{InCone}

    \SetKwProg{Fn}{Function}{:}{}
    \Fn{\FMain{$\sigma_{\operatorname{part}}, \poset,V,{\tt Ret}$}}{
    
    \If{$\sigma_{\operatorname{part}} ==[\text{ }]$ and \FCheckCone{V} is not \textbf{True}}{
        Return 
    }
    
    $l+1$ = length of $\sigma_{\operatorname{part}}$
    
    \If{l == K}
     {add  $\sigma_{\operatorname{part}}$ to ${\tt Ret}$\; Return\;}
    
    \For{i = 0 .. K}{
        \If{$i \not\in \sigma_{\operatorname{part}}$}{
                $\RV' = \RV_{p_i} - \RV_{p_{\sigma_{\operatorname{part}}(l)}}$
                
     \If{not \InCone{$-\RV',V$}}{\FMain{$\sigma_{\operatorname{part}}+[i],  \poset, V \cup\{v'\}, {\tt Ret}$}\;}
            
        }
    
     }
    }
\textbf{End Function}
\end{algorithm}

In the Algorithm \ref{algorithm:core algorithm}, for convenience, we take $\RV_{p_{\sigma_{\operatorname{part}}(-1)}}=0$. To prove the correctness of the algorithm it is useful to denote the return of Algorithm~\ref{algorithm:core algorithm}  given input $(\poset, \prec, \Xi)$ as $\totord_{alg}(\poset, \prec, \Xi)$.

\begin{definition}
For fixed $(\poset, \prec,\Xi)$, $\sigma\in S_{K+1}$ and for $k=1,\cdots, K$, define 
\[
V_{\sigma,k} = \{\RV_{p_{\sigma(i)}} - \RV_{p_{\sigma(i-1)}}\}_{i=1,\dots,k}\cup V_0,
\]
where $V_0$ is the base cone for $(\poset, \prec,\Xi)$ as in Definition \ref{definition:V0}, and $\RV_{p_{\sigma(j)}}, j = 0, \ldots,K$ is the representation vector of $p_j\in \poset$. For convenience, we define $V_{\sigma, 0} = V_0$
and we observe that $ V_{\sigma, K} = V_{\sigma}$ from Equation \eqref{eq:Vord}. 
\end{definition}

\begin{theorem}
\label{thm:algorithm_correctness}
 Algorithm \ref{algorithm:core algorithm} solves the LC-LEP.
\end{theorem}
\begin{proof}
 Given $(\poset, \prec, \Xi)$, we need to show that $\totord(\poset, \prec, \Xi) = \totord_{alg}(\poset, \prec, \Xi)$. 
We may assume that  $\cone(V_0)$ is pointed since if not, then both $\totord_{alg}(\poset, \prec, \Xi)$ and $\totord(\poset, \prec, \Xi)$ are empty.

We first show that $\totord_{alg}(\poset, \prec, \Xi) \subset \totord(\poset, \prec, \Xi)$, i.e.\ for any $\sigma\in \totord_{alg}(\poset, \prec, \Xi)$ we show that the set $\Xi_{\sigma}\neq \emptyset$. 
As indicated above we assume $\cone(V_0) = \cone(V_{\sigma,0})$ is pointed. For Algorithm \ref{algorithm:core algorithm}, lines 2-4 returns the empty set if $\cone(V_0)$ is not pointed. 
Otherwise, it passes to lines 5-9 which checks if $\sigma_{\operatorname{part}}$ is a total order over $\{0,\dots, K\}$. 
If so, it is added to the return variable, {\tt Ret}. 
If $\sigma_{\operatorname{part}}$ is not a total order, then lines 10-17 extend it to a total order by recursively constructing $V_{\ord, i}$ from $V_{\ord, i-1}$ for $1 \leq i \leq K$. 

Therefore, it suffices to show that $V_{\sigma,k}$ are all pointed for $k=1,\dots, K$.

Fix $k \in \setof{1, \dotsc, K}$.
In lines 11-12, we find a candidate $i\in\{0,\dots,K\}$ which is not in the image of $\sigma_{\operatorname{part}}$, and define $V_{\sigma, k+1} = V_{\sigma, k} \cup \{\RV'\}$ where $\RV' = \RV_{p_i} - \RV_{p_{\sigma(k)}}$. In line 13, we verify that $-\RV'\notin V= V_{\sigma,k-1}$ and it follows from Proposition \ref{prop:cone_extenstion}, that $\cone(V_{\sigma,k})$ is pointed. For each $\sigma$ appended to {\tt Ret}, we have $\cone(V_{\sigma,k})$ is pointed for $k=0,\dots, K$. In particular, $\cone(V_{\sigma, K}) = \cone(V_{\sigma})$ is pointed, and from Proposition \ref{prop:order_iff_pointed_cone}, we have $\Xi_{\sigma} \neq \emptyset$.

We now prove that $\totord(\poset, \prec, \Xi) \subset \totord_{alg}(\poset, \prec, \Xi)$. 
Assume that $\sigma\in \totord(\poset, \prec, \Xi)$. By definition this means that $\Xi_{\sigma} \neq \emptyset$ and from Proposition \ref{prop:order_iff_pointed_cone}, $V_{\sigma}$ is pointed. 
For each $k =1,\dots, K$, we have $V_{\sigma,k}\subset V_{\sigma}$, and thus $V_{\sigma,k}$ is pointed. As $V_{\sigma,k}$ is pointed, we know $-(\RV_{p_{\sigma(k)}}-\RV_{p_{\sigma(k-1)}})\notin V_{\sigma,k-1}$ for $k=1,\dots,K$. Therefore, line 13 in Algorithm  \ref{algorithm:core algorithm} will not fail to extend $\sigma$ at each step in the recursion and after $K$ recursive extensions, $\sigma$ will be appended to {\tt Ret} and thus, $\sigma\in \totord_{alg}(\poset, \prec, \Xi)$.
\end{proof}

\subsection{Complexity analysis}
\label{subsection:complexity}

As discussed in Section \ref{sec:introduction}, computing linear extensions of a poset is at least as hard as counting them, which has been established to be a difficult problem. While the algebraic constraints induced by the evaluation domain reduces the number of admissible linear extensions, it is not clear whether or not the constrained problem is easier (i.e.\ whether or not the LC-LEP is also $\#P$-complete). 
The answer may depend on the partial order, the evaluation domain, or both. 
In any case, this appears to be a difficult and open problem which is outside the scope of this work. Consequently, we do not attempt to provide any rigorous asymptotic estimates of Algorithm \ref{algorithm:core algorithm}.

However, for the instances of LC-LEP solved in this work, our algorithm is dramatically more efficient than a brute force algorithm as evidenced by the results in Section \ref{sec:results}. 
By brute force we mean an algorithm which checks every linear order on $\cP$ to determine whether or not it is admissible. 
Thus we conjecture that Algorithm \ref{algorithm:core algorithm} is more efficient in the typical case which explains our results. 

A heuristic, but not rigorous, justification is as follows. We start by assuming a fixed implementation for the {\tt InCone} function in Algorithm \ref{algorithm:incone}.
This is the core algorithm in the sense that it dominates the computational cost of Algorithm \ref{algorithm:core algorithm}. 
Let $(\cP, \prec, \evd)$ be a given instance of the LC-LEP where $\poset = \setof{p_0, \dotsc, p_K} \subset \rr[x_1,\dots, x_d]$ and let $\cT(\cP, \prec, \evd) \subseteq S_{K+1}$ denote the solution of the LC-LEP.

We consider the number of {\tt InCone} function calls required to solve this instance (i.e.~we suppose that each call to {\tt InCone} has unit computational cost). It is trivial to count the number of calls to {\tt InCone} necessary for a brute force search. By Proposition \ref{prop:order_iff_pointed_cone}, $\evd_\sigma$ is nonempty for any $\sigma \in S_{K+1}$ if and only $\cone(V_\sigma)$ is pointed which is determined by a single {\tt InCone} call. Therefore, a brute force search recovers $\cT(\cP, \prec, \evd)$ in exactly $(K+1)!$ calls to {\tt InCone}. 

To contrast this with Algorithm \ref{algorithm:core algorithm}, consider a fixed $\sigma \in S_{K+1}$ which is {\em not} admissible. Hence, there exists a least index, $i \in \setof{0,\dotsc, K}$, such that 
\[
\sigma' \big|_{\setof{0,\dotsc,i-1}} = \sigma \big|_{\setof{0,\dotsc,i-1}}\qquad \text{for at least one } \sigma' \in \cT(\cP, \prec, \evd)
\]
and
\[
\sigma' \big|_{\setof{0,\dotsc,i}} \neq \sigma \big|_{\setof{0,\dotsc,i}}\qquad \text{for any } 
\sigma' \in \cT(\cP, \prec, \evd).
\]
Consequently, in the $i^{\rm th}$ recursive call, line $13$ of Algorithm \ref{algorithm:core algorithm} returns {\tt False} and $\sigma$ is ``pruned'' (i.e.~removed from consideration as a candidate solution). In other words, $\sigma$ has already been ruled incompatible with the algebraic constraints imposed by $\prec$ and $\evd$ using only its partial image obtained by restriction to the subset $\setof{0, \dotsc, i}$. The key observation is that this applies to every such incompatible linear extension. In other words, every $\sigma' \in S_{K+1}$ satisfying 
\[
\sigma' \big|_{\setof{0,\dotsc,i}} = \sigma \big|_{\setof{0,\dotsc,i}}
\]
is also pruned in this step. Evidently, there are $(K-i)!$ such extensions which are simultaneously removed which results in saving $(K-i)!$ calls to {\tt InCone} compared to the brute force search.

Pruning incompatible partial images early leads to an exponential reduction in the number of {\tt InCone} calls which explains the results shown in Table \ref{table:results}. However, a more rigorous analysis for the complexity of Algorithm \ref{algorithm:core algorithm} amounts to estimating how early a given incompatible partial image is pruned. It is likely that this is heavily dependent on either $\prec$ or $\evd$, or both, and may also depend on the order in which the partial images are extended. 

\section{Solving the general PSD problem}
\label{sec:PSD}

In this section we present a solution for the PSD problem described in  Section~\ref{sec:introduction}. 
The solution is based on the observation that the PSD problem naturally has a Boolean lattice structure. 
Therefore, the linear PSD problem is an instance of the LC-LEP and and reduces to a simple application of the algorithms presented in Section \ref{sec:LCLEP}.
For nonlinear PSD problems, we construct a map that ``embeds'' it into an instance of LC-LEP (of higher dimension) in the sense that the inclusion in Equation \eqref{eq:PSD_solution_inclusion} holds for the Boolean partial order. We prove a sufficient condition for which this inclusion is equality and describe a method for disqualifying spurious solutions when it is strict.

\subsection{The PSD as an instance of AC-LEP}
\label{sec:psd_is_boolean_lattice}

Throughout this section $\left(\poset, \prec_B, (0,\infty)^{2n}\right)$ denotes a PSD problem for a fixed interaction function $f$ of order type $\PL \in \nn^q$ as defined in Equation \eqref{eq:polynomials} where $\prec_B$ is the Boolean lattice partial order. 
Our first goal is to show that $\psdR$ satisfies Equation \eqref{eq:primaryPO}, and in particular, that every $\sigma \in \totord \psdR$ is a linear extension of a Boolean lattice. 
We start by defining appropriate indices for the elements of $\poset$. 
\begin{definition}
\label{def:Boolean_functional_indexing}
Suppose $\mathbf{n} \in \nn^q$ is the interaction type for $f \in \rr\left[z_1,\dotsc, z_n\right]$. 
As in Definition~\ref{def:interaction_function} let $\setof{I_1,\dotsc,I_q}$ denote the indexing sets for each summand of $f$ which form an integer partition of $I := \setof{1,\dotsc, n}$. For each $1\leq j \leq q$, we consider $I_j$ as an ordered set and and for $1 \leq k \leq n_j$, we let $I_j(k)$ denote the $k^{\rm th}$ largest element of $I_j$. Let $E := \setof{\alpha : \setof{1,\dotsc, n} \to \setof{0,1}}$ be the set of all Boolean functions defined on $I$.
The {\em Boolean indexing map}, denoted by $B: E \to \setof{0,\dotsc, 2^{n-1}}$, is defined by the formula
\[
B(\alpha) := \sum_{j = 1}^q \sum_{k=1}^{n_j}  \alpha(I_j(k))2^{n - I_j(k)}.
\]
In other words, $E$ and $B$ are defined so that for each $d \in I$, $\alpha(d)$ is the $d^{\rm th}$ binary digit of $B(\alpha)$ in Little-endian order.
\end{definition}

We will also consider, $\alpha \in E$, as a vector of Boolean functions defined as follows. 
Let $E_j$ denote the set of Boolean functions defined on $I_j$. 
Then, elements of $E$ can be represented as vectors of the form
\[
\alpha = \left(\alpha_1, \dotsc, \alpha_q \right) \quad \text{where} \ \alpha_j := \alpha \Big|_{I_j} \in E_j \quad \text{for} \  1 \leq j \leq q.
\]
Note that under this identification, $E$ has the equivalent representation as $E = E_1 \times \dots \times E_q$.

\begin{definition}
\label{def:e_polynomial} 
Suppose $\mathbf{n} \in \nn^q$ is the interaction type for an interaction function, $f \in \rr\left[z_1,\dotsc, z_n\right]$ as in Definition \ref{def:interaction_function}, and  $E$ denotes the corresponding Boolean indices. 
For $\alpha \in E$, define $p_\alpha \in \poset \subset \rr \left[\ell_1,\dotsc, \ell_n, \delta_1, \dotsc, \delta_n\right]$, by the formula
\begin{equation}
\label{eq:Boolean_idx_on_evalpoly}
	p_\alpha := \prod_{j = 1}^{q} \left( \sum_{k \in I_j} \ell_{k} + \alpha(k) \delta_{k} \right). 
\end{equation}
\end{definition}

When convenient, we  use a linear indexing scheme for elements of $\poset$ which we define via the Boolean indexing map by identifying $p_i := p_\alpha$ where $\alpha = B^{-1}(i)$. 
To avoid confusion, we  exclusively use Greek subscripts when referring to elements of $\poset$ by their Boolean indices, and Latin subscripts when referring to elements of $\poset$ by their linear indices. 
We leave it to the reader to check that the linearly indexed polynomials in Examples~\ref{ex:3input} and \ref{ex:3inputLin} are in agreement with that of Definition~\ref{def:e_polynomial} via this identification.

\begin{definition}
\label{def:one_bit_condition}
Let $\alpha,\beta \in E$ be a pair of Boolean indices corresponding to $\PL \in \nn^q$. 
An ordered pair $(\alpha, \beta)$ satisfies the \emph{one bit condition} if $\alpha(I_j(k))\leq \beta(I_j(k))$, for all $1 \leq j \leq q$ and $0 \leq k \leq n_j-1$, with equality for all but {\em exactly} one $(j,k)$ pair.
\end{definition}

\begin{remark}
Observe that if $(\alpha,\beta)$ satisfy the one bit condition and $(j_0,k_0)$ is the unique pair for which $\alpha$ and $\beta$ take different values, then $\alpha(I_{j_0}(k)) = 0$ and $\beta(I_{j_0}(k)) = 1$. 
\end{remark}

\begin{remark}
\label{rem:1bitposet}
The one bit condition induces a poset structure on $E$ by setting $\alpha \prec \beta$ for each $(\alpha,\beta)$ satisfying the one bit condition, and extending the relation transitively. The one bit condition is a covering relation for the partial order. 
\end{remark}

\begin{definition}
\label{defn:PSDProblem}
Suppose $\mathbf{n} \in \nn^q$ is the interaction type for an interaction function, $f \in \rr\left[z_1,\dotsc, z_n\right]$ as in Definition \ref{def:interaction_function}, and  $E$ denotes the corresponding Boolean indices. Let $\poset$ be the set of polynomials indexed as in Definition~\ref{def:e_polynomial}. 
The \emph{PSD problem} is defined by the triple, $(\poset, \prec, (0,\infty)^{2n})$ where $\prec$ is given by Remark~\ref{rem:1bitposet}.
\end{definition}

The next proposition proves that $(\poset, \prec, (0,\infty)^{2n})$ satisfies Equation~\eqref{eq:primaryPO}, and furthermore that $(\poset, \prec)$ is a Boolean partial order which justifies expressing the PSD problem as $\psdR$.

\begin{proposition}
\label{prop:P_is_boolean} 
Consider a PSD problem $(\poset, \prec, (0,\infty)^{2n})$.
Then, 
\begin{enumerate}
    \item $(\poset, \prec)$ is a Boolean lattice. 
    \item For any $\alpha, \beta \in E$, if $\alpha \prec \beta$, then 
    \[
    p_\alpha(\parm) < p_\beta(\parm) \quad \text{for all} \quad \parm \in (0,\infty)^{2n}. 
    \]
\end{enumerate}
\end{proposition}

\begin{proof}
To prove the first claim, let $S_n = \setof{1, \dotsc, n}$ and let $(2^{S_n}, \prec_B)$ denote the standard Boolean lattice. Define a map, $\varphi : E \to 2^{S_n}$, by the formula
\[
\varphi(\alpha) = \setof{j \in S_n : \alpha(j) = 1},
\]
and we note that $\varphi$ is a bijection since $E$ is defined to be the collection of all Boolean maps defined on $S_n$. Furthermore, for any $\alpha, \beta \in E$, we have by Definition \ref{def:one_bit_condition} that 
$\alpha \prec \beta$ if and only if 
\[
\setof{j \in S_n : \alpha(j) = 1} \subset \setof{j \in S_n : \beta(j) = 1}
\]
implying that $\varphi$ is an order isomorphism. 

To establish the second claim, we must show that if $\alpha \prec \beta$, then $p_\alpha(\parm) < p_\beta(\parm)$ holds for all $\parm \in (0,\infty)^{2n}$. Note that by transitivity, it suffices to prove this holds for $(\alpha, \beta)$ satisfying the one bit condition. In this case we have
\[
\beta(I_j(k)) - \alpha(I_j(k)) = 
\begin{cases}
1 & \text{if} \ j = j_0 \text{ and } k = k_0 \\
0 & \text{otherwise} .
\end{cases}
\]
for some $j_0 \in \setof{1,\dotsc, q}$, $k_0 \in \setof{0,\dotsc,n_{j_0-1}}$. If $\parm = \left(\ell_1,\dotsc, \ell_n, \delta_1, \dotsc, \delta_n \right) \in \evd$, then from Equation \eqref{eq:Boolean_idx_on_evalpoly} we have
\begin{align*}
	p_\beta(\parm) & = \left( \left( \sum_{k \in I_{j_0}} \ell_k + \alpha(I_j(k)) \delta_k \right) + \delta_{k_0} \right) \prod_{j \neq j_0}\sum_{k \in I_j} \ell_k + \alpha(I_j(k)) \delta_k \\
	& = p_\alpha(\parm) + \delta_{k_0} \prod_{j \neq j_0}\sum_{k \in I_j} \ell_k + \alpha(I_j(k)) \delta_k \\
	& > p_\alpha(\parm)
\end{align*}
as required. 
\end{proof}
With Proposition \ref{prop:P_is_boolean} in mind, we return to writing $\prec_B$ in place of $\prec$ for the PSD problem where $\prec_B$ is the partial order of a Boolean lattice inherited by $\poset$ from the one bit condition.

\subsection{The linear PSD problem}
\label{sec:linear_PSD}
We consider two cases of the PSD problem: the interaction type $\mathbf{n} \in \nn^q$ for the function
$f \in \rr\left[z_1,\dotsc, z_n\right]$ has the form $\mathbf{n}=(n)$ or $\mathbf{n}=(1,1,\ldots,1)$. In the first case, $f$ is linear (see Example~\ref{ex:3inputLin}) and the PSD problem is an instance of LC-LEP. In the second case, $\log f$ is linear so after a simple change of variables, we obtain an instance of LC-LEP with equivalent solutions since $\log$ is monotone, hence order preserving. We focus on the first case, leaving it to the reader to check that the second case is same modulo the evaluation domain ($\rr^{2n}$ versus $(0,\infty)^{2n})$.
Following the algorithm described in Section \ref{sec:LCLEP}, we encode the partial order defined by $\prec_B$ as a set of linear constraints defined by a base cone which we must show is pointed. We begin by denoting the set of representation vectors for $\poset$ as
\[
\mathcal{V} := \setof{\RV_{p_{\alpha}} \in \setof{0,1}^{2n} : \RV_{p_{\alpha}} \ \text{is the representation vector of} \ p_\alpha, \ \alpha \in E }.
\]
We define the set, 
\begin{equation}
\label{eq:linear_PSD_base_cone}
V_{\prec_B} := \setof{\RV_{p_{\beta}} - \RV_{p_{\alpha}} :\RV_{p_{\alpha}},\RV_{p_{\beta}}\in \mathcal{V}, \ (\alpha, \beta) \ \text{satisfies the one bit condition}}.
\end{equation}
which encodes the $\prec_B$ partial order into the representation vectors. These vectors will be the generators of the base cone for the algorithm in Section \ref{sec:LCLEP}. Thus, we must show that $\cone(V_{\prec_B})$ generates a pointed cone. 
\begin{lemma}
\label{lem:onebit_cone_is_pointed}
Let $\cob := \cone(V_{\prec_B})$ denote the cone generated by $V_{\prec_B}$, then $\cob$ is pointed. 
\end{lemma}
\begin{proof}  
By Proposition~\ref{prop:cccone}, $\cob$ is closed and convex so it suffices to prove that if $v \in \cob$ and $-v \in \cob$, then $v = 0$. Fix $\parm \in (0, \infty)^{2n}$ and suppose $(\alpha,\beta)$ satisfies the one bit condition. By the formula in Equation \eqref{eq:Boolean_idx_on_evalpoly} it follows that 
\[
p_\beta - p_\alpha = \delta_i
\]
for some $i \in \setof{1,\dotsc, n}$. Since $\delta_i = \parm_{n + i} > 0$ for all $\parm \in \evd$, it follows that 
\[
p_\beta(\parm) - p_\alpha(\parm) > 0,
\]
for every $(\alpha,\beta)$ satisfying the one bit condition. Passing to the representation vectors, it follows that for every $v \in V_{\prec_B}$, we have $v \cdot \parm > 0$.
Taking the conic hull, we have that if $v \in \cob\setminus\setof{0}$, then $v \cdot \parm > 0$. 
It follows that if $v,-v \in \cob$ simultaneously, then $v \cdot \parm \geq 0$ and $-v \cdot \parm \geq 0$ implying $v = 0$. 
\end{proof}

\subsection{The general PSD problem}
\label{sec:general_PSD}

Given a general PSD problem $\psdR$ we present the construction of a LC-LEP  denoted by $(\poset', \prec_B, \rr^m)$ with the property that  $\totord\psdR \subseteq \totord(\poset', \prec_B, \rr^m)$.
The importance of this is that $(\poset', \prec_B, \rr^m)$ can be solved using Algorithm~\ref{algorithm:core algorithm} and hence we obtain a rigorous upper bound on $\totord\psdR$.

\begin{definition} 
\label{def:linearization_map}
Given an interaction type $\mathbf{n} \in \nn^q$, let $E = E_1 \times \dots \times E_q$ denote the corresponding Boolean indices. 
Set $m := \sum_{j=1}^q 2^{n_j}$ and  define the {\em linearized evaluation domain} to be 
\begin{equation}
\label{eq:linearized_evd}
\rr^{2^{n_1}} \times \dots \times \rr^{2^{n_q}} \cong \rr^m.
\end{equation}
Define a polynomial ring in $m$ indeterminates with Boolean indexing by
\begin{equation}
	\label{eq:linearized_polynomial_ring}
	\cR := \rr \left[\setof{x_{j,\alpha_j} : \alpha_j \in E_j, 1 \leq j \leq q}\right],
\end{equation}
and define a collection of linear polynomials by
\[
\poset' := \setof{p'_\alpha : \alpha \in E} \subset \cR \qquad \text{where} \quad p'_\alpha :=  \sum_{j=1}^q x_{j,\alpha_j}.
\]
The \emph{linearized PSD problem} determined by $\mathbf{n}$ is to compute $\totord\left(\poset', \prec_B, \rr^m\right)$.
\end{definition}

\begin{theorem}
	\label{thm:P'_is_boolean}
Fix an interaction type, $\mathbf{n} \in \nn^q$ and let $\totord\psdR$ and $\totord\lpsdR$ denote the corresponding PSD and linearized PSD problems, respectively. 
The following are true.
\begin{enumerate}
    \item Let $\alpha, \beta \in E$ and $\parm \in (0,\infty)^{2n}$. 
    If $p_\alpha(\parm) <  p_\beta(\parm)$ and 
    \[
    \parm'_{j, \alpha_j} = \log{ \paren{\sum_{k \in I_j} \parm_{k} + \alpha_j(k) \parm_{n + k}}} \in \rr^m,
    \] then $p'_\alpha(\parm') < p'_\beta(\parm')$.
	\item $\totord \psdR \subseteq \totord \lpsdR$.
\end{enumerate}  
\end{theorem}

\begin{proof}
To prove the first claim, we define a map, $T: (0,\infty)^{2n} \to \rr^m$, by $\parm \mapsto \parm' := T(\parm) $ where the coordinates of $\parm'$ are given by the formula
\begin{equation}
\label{eq:definition_of_T}
    \parm'_{j, \alpha_j} = \log{ \paren{\sum_{k \in I_j} \parm_{k} + \alpha_j(k) \parm_{n + k}}}.
\end{equation}
Observe that $T$ is defined to satisfy the functional equation
\begin{equation}
\label{eq:T_functional_identity}
   \log \circ p_\alpha(\parm) = p'_\alpha \circ T(\parm) \qquad \text{for all} \quad \alpha \in E, \ \parm \in (0,\infty)^{2n}.  
\end{equation}
Therefore, if $\alpha, \beta \in E$ and $\parm \in(0,\infty)^{2n}$ satisfies $p_\alpha(\parm) < p_\beta(\parm)$, then $\log \left(p_\alpha(\parm)\right) < \log \left(p_\beta(\parm)\right)$ and it follows from Equation \eqref{eq:T_functional_identity} that $p'_\alpha(\parm') < p'_\beta(\parm')$ where $\parm' = T(\parm)$ as required.  

To prove the second claim, consider $\poset$ and $\poset'$ equipped with the linear indices as in Definition \ref{def:e_polynomial}, and suppose $\ord \in \totord \psdR$. Then, by definition there exists  $\parm \in (0,\infty)^{2n}$ satisfying
\[
p_{\sigma(0)}(\parm) < p_{\sigma(1)}(\parm) < \dots < p_{\sigma(2^n-1)}(\parm).
\]
Let $\parm' = T(\parm)$ and apply the first result to successive pairs in the ordering which implies that for all $0 \leq k \leq 2^n-2$, we have
\[
p'_{\sigma(k)}(\parm') = \log \left(p_{\sigma(k)}(\parm)\right) < \log \left( p_{\sigma(k+1)}(\parm) \right) = p'_{\sigma(k+1)}(\parm').
\]
Thus, we $\parm' \in \rr^m$ satisfies
\[
p'_{\sigma(0)}(\parm') < p'_{\sigma(1)}(\parm') < \dots < p'_{\sigma(2^n-1)}(\parm'),
\]
and it follows that $\sigma \in \totord \lpsd$ which completes the proof. 
\end{proof}

\begin{example}
Recall the PSD in Example \ref{ex:3input} with interaction function $f(z) = (z_1 + z_2)z_3$ corresponding to interaction type $\PL = (2,1)$. The corresponding polynomials for the linearized PSD problem are the following polynomials in $m = 6$ variables
\begin{align*}
p'_0 & = x_{1,(0,0)} + x_{2,0} & p'_4 & = x_{1,(0,1)} + x_{2,0} \\
p'_1 & = x_{1,(0,0)} + x_{2,1} & p'_5 & = x_{1,(0,1)} + x_{2,1} \\
p'_2 & = x_{1,(1,0)} + x_{2,0} & p'_6 & = x_{1,(1,1)} + x_{2,0} \\
p'_3 & = x_{1,(1,0)} + x_{2,1} & p'_7 & = x_{1,(1,1)} + x_{2,1}
\end{align*}
where we have used linear indexing for elements of $\cP$ to match the polynomials in Example \ref{ex:3input}.

As defined in Equation \eqref{eq:linearized_polynomial_ring}, each variable of the form, $x_{1, (x,y)}$, denotes an indeterminate in $\cR$ which is identified with the element, $\alpha_1 \in E_1$, which satisfies $\alpha_1(1) = x$ and $\alpha_1(2) = y$. In other words, the binary vector subscript, $(x,y)$, denotes the two values which $\alpha_1 \in E_1$ takes for the inputs in $I_1 = \setof{1,2}$. Similarly, $x_{2,x}$ is identified with $\alpha_2 \in E_2$ satisfying $\alpha_2(3) = x$. 
\end{example}

\subsection{Solving the PSD problem for interaction type $\mathbf{n}=(2,1,\ldots,1)$.}
\label{sec:special_case}
In this section we prove the following theorem.
\begin{theorem}
\label{thm:characterization_of_equivalent_orders}
Let $f$ be an interaction function with interaction type, $\mathbf{n}=(2,1,\ldots,1)$. 
Let $\totord\psdR$ denote the corresponding PSD problem and $\lpsdR$ the associated linearized PSD problem. Then $\totord\psdR = \totord(\poset', \prec_B, \Xi')$ where $\Xi' = \rr^m \cap \{-\xi'_{1,0}+\xi'_{1,1} + \xi'_{1,2}- \xi'_{1,3} > 0\}$.
\end{theorem}
 
The proof of the theorem is based on the following lemma

\begin{lemma}
\label{lem:root_exponential_polynomial}   
Fix parameters, $x_0,x_1,x_2,x_3 \in \rr$, and define the function, $g: \rr \to \rr$ by the formula
\[
g(t) = \exp(tx_0) - \exp(tx_1) - \exp(tx_2) + \exp(tx_3).
\]
If $x_0 < x_1 \leq x_2 < x_3$, then $g$ has a positive root if and only if $g'(0) < 0$. 
\end{lemma}
\begin{proof}
Suppose first that $t_0$ is a root of $g$. Expanding $\exp(t_0 x_1)$ and $\exp(t_0 x_2)$ to first order about $x_0$ and $x_3$, respectively, and applying the mean value theorem yields the formula
\begin{equation}
\label{eq:apply_MVT}
g(t_0) = -t_0 \exp(t_0 c_1)(x_1 - x_0) - t_0 \exp(t_0 c_2)(x_2 - x_3) = 0
\end{equation}
for some $c_1 \in (x_0,x_1)$ and $c_2 \in (x_2,x_3)$. We define $k = c_2 - c_1$ and multiply Equation \eqref{eq:apply_MVT} by $t_0 e^{-k t_0}$ to obtain
\[
e^{k t_0}(x_3 - x_2) - (x_1 - x_0) = 0.
\]
Noting that $c_1 < x_2 < c_2$, it follows that $k > 0$. Therefore if $t_0 > 0$, then $x_3 - x_2 < x_1 - x_0$ or equivalently, $g'(0) = x_0 - x_1 - x_2 + x_3 < 0$.

Conversely, if $g'(0) < 0$ then $g$ has at least one positive root since clearly $g(0) = 0$ and $\lim\limits_{t \to \infty} g(t) = \infty$. 
\end{proof}

\begin{proof}[Proof of Theorem~\ref{thm:characterization_of_equivalent_orders}]

Suppose $\sigma\in \totord\psdR$, $(0,\infty)^{2n}_{\sigma}$ is the associated realizable set as defined in Equation \eqref{eq:xisigma}, and $\xi \in (0,\infty)^{2n}_{\sigma}$, then by Theorem \ref{thm:P'_is_boolean} we have $\xi' = T(\xi) \in \rr^m_\sigma$. Note that by definition the first four coordinates of $\xi'$ are given by the formulas
\begin{align*}
    \xi'_{1,0} & = \log (\xi_1 + \xi_2) \\
    \xi'_{1,1} & = \log (\xi_1 + \xi_2 +\xi_{n+2}) \\
     \xi'_{1,2} & = \log (\xi_1 + \xi_2 +\xi_{n+1}) \\
    \xi'_{1,3} & = \log (\xi_1 + \xi_2 +\xi_{n+1} +\xi_{n+2}).
\end{align*}
Since $\xi_i > 0$ for $i \in \setof{1,2,n+1,n+2}$, it follows that
\[-\xi'_{1,0}+\xi'_{1,1} + \xi'_{1,2} - \xi'_{1,3} > 0,\] so we have $\sigma \in \totord(\poset', \prec_B, \Xi' )$.

Conversely, suppose $\sigma \in \totord(\poset', \prec_B, \Xi')$ and $\xi'\in\Xi'_{\sigma}$.
From the Boolean lattice  $\prec_B$ we have $\xi'_{1,0} < \xi'_{1,1} \leq \xi'_{1,2} < \xi'_{1,3}$ or $\xi'_{1,0} < \xi'_{1,2} \leq \xi'_{1,1} < \xi'_{1,3}$. Moreover, $\xi'$ also satisfies, $-\xi'_{1,0} + \xi'_{1,1} + \xi'_{1,2} - \xi'_{1,3} >0$. Hence, Lemma \ref{lem:root_exponential_polynomial} implies that there exists $t' > 0$ such that $\hat{\xi}' := t'\xi'$ satisfies
\[
 \exp( \hat{\xi}'_{1,0}) - \exp( \hat{\xi}'_{1,1}) - \exp( \hat{\xi}'_{1,2}) + \exp( \hat{\xi}'_{1,3})=0.
\]

Next, we define $\hat{\parm} \in (0,\infty)^{2n}$ by
\[
\ \hat{\parm}_j = 
\begin{cases}
\exp(\hat{\parm'}_{j,0}) & 2 < j \leq n \\
\exp(\hat{\parm'}_{j,1}) - \exp(\hat{\parm'}_{j,0}) & n+2 < j < 2n \\
\frac{1}{2}\exp(\hat{\parm'}_{1,0}) & j = 1,2 \\
\exp(\hat{\parm'}_{1,2}) - \exp(\hat{\parm'}_{1,0}) & j = n+1 \\
\exp(\hat{\parm'}_{1,1}) - \exp(\hat{\parm'}_{1,0}) & j = n+2 \\
\end{cases}
\]
One easily verifies that $\hat{\parm}_j > 0$ for all $1 \leq j \leq 2n$, and that $T(\hat{\parm}) = \hat{\parm'}$. From Theorem \ref{thm:P'_is_boolean},  we have $\hat{\parm} \in (0,\infty)^{2n}_{\ord}=\{\xi \in (0,\infty)^{2n} :p_{\sigma(0)}(\parm) < p_{\sigma(1)}(\parm) < \dots < p_{\sigma(2^n-1)}(\parm)\}$ which implies that $\sigma \in \totord\psdR$.
\end{proof}


\subsection{Solving the general PSD problem}
\label{sec:SolvingPSDGeneral}
In the general case, we have $\totord \left(\poset, \prec_B, (0,\infty)^{2n} \right) \subsetneq \totord \lpsd$, and thus, computing $\totord \left(\poset', \prec_B, \rr^m \right)$ provides only a set of candidates for $\totord \left(\poset, \prec_B, (0,\infty)^{2n} \right)$. This candidate set contains spurious linear extensions so we consider the problem of removing linear extensions which are non-admissible. We have two strategies for doing this efficiently. 

The first is to restrict the evaluation domain to a strict subset, $\Xi' \subsetneq \rr^m$, such that we still have the inclusion
\begin{equation}
\label{eq:additional_linear_constraints}
\totord \left(\poset, \prec_B, (0,\infty)^{2n} \right) \subseteq  \totord(\poset',\prec_B, \Xi').
\end{equation}
Restricting to a smaller evaluation domain amounts to imposing more of the algebraic constraints a-priori which results in improved efficiency. In order for the candidate set on the right hand side to be efficiently computable using the algorithm in Section \ref{sec:LCLEP}, it must be an instance of the LC-LEP i.e.~$\evd'$ should be the interior of a polyhedral cone. For example, for the PSD with interaction type  $\textbf{n} = (2,1,\dotsc,1)$, analyzed in Section \ref{sec:special_case}, we computed on the restricted domain
\[
\Xi' = \rr^m\cap \setof{\xi' \in \rr^m : -\xi'_{1,0}+\xi'_{1,1} + \xi'_{1,2}- \xi'_{1,3} > 0}.
\] In terms of the algorithm in Section \ref{sec:LCLEP}, this domain restriction amounts to taking our base cone in Algorithm \ref{algorithm:core algorithm} to be $\cone(V_0)$ where 
\[
V_0 = V_{\prec_B} \cup \setof{\RV}
\]
and $\RV$ is the representation vector for the linear functional defined by the formula 
\[
x \mapsto -x_{1,0}+x_{1,1}+x_{1,2} - x_{1,3}.
\] The requirement that this linear functional must be strictly positive is a special case of the following Lemma whose proof is a trivial computation.
\begin{lemma}
\label{lem:simple_calculus_ineq}
Suppose $\alpha, \alpha', \beta, \beta'$ are Boolean indices such that for any $\parm \in (0, \infty)^{2n}$, the following equations are satisfied.
\begin{eqnarray*}
p_{\alpha}(\parm) < p_{\beta}(\parm) < p_{\beta'}(\parm) < p_{\alpha'}(\parm) \\
p_{\alpha}(\parm) + p_{\alpha'}(\parm) = p_{\beta}(\parm) + p_{\beta'}(\parm).
\end{eqnarray*}
Then, 
\[
\log(p_{\alpha}(\parm)) + \log(p_{\alpha'}(\parm)) - \log(p_{\beta}(\parm)) - \log(p_{\beta'}(\parm)) > 0.
\]
\end{lemma}
Lemma \ref{lem:simple_calculus_ineq} provides a means to restrict the evaluation domain for the general linearized PSD problem as follows. Fix $j \in \setof{1,\dotsc,q}$ and suppose $\setof{\alpha, \alpha', \beta, \beta'} \subset E$ differ only in the $j^{\rm th}$ coordinate with $\alpha \prec_B \beta \prec_B \beta' \prec_B \alpha'$, and also assume that $B(\alpha) + B(\alpha') = B(\beta) + B(\beta)'$ where $B$ is the Boolean indexing map. Then, it follows that for any $\parm \in \evd$, the values, $\setof{p_\alpha(\parm), p_{\alpha'}(\parm), p_\beta(\parm), p_{\beta'}(\parm)}$, satisfy both equations in Lemma \ref{lem:simple_calculus_ineq}. Therefore, if $\RV(\setof{\alpha, \alpha', \beta, \beta'})$ is the representation vector for the linear functional defined by 
\[
x \mapsto x_{j,B(\beta)} + x_{j, B(\beta')} - x_{j, B(\alpha)} - x_{j, B(\alpha')},
\]
then $v(\setof{\alpha, \alpha', \beta, \beta'})$ lies in $V_\sigma$ for any $\sigma \in \totord\left(\poset, \prec_B, (0, \infty)^{2n}\right)$.  Equivalently, we may impose the required linear constraint, $x_{j,B(\beta)} + x_{j, B(\beta')} - x_{j, B(\alpha)} - x_{j, B(\alpha')} > 0$ on the evaluation domain of the linearized problem. Hence, for each $1 \leq j \leq q$, we define 
\[
V_j := \setof{\RV(\setof{\alpha, \alpha', \beta, \beta'}) : B(\alpha) + B(\alpha') = B(\beta) + B(\beta)', \  \alpha \prec_B \beta \prec_B \beta' \prec_B \alpha'}
\]
and for an arbitrary PSD problem, we may take our base cone to be
\[
V_0 = V_{\prec_B} \cup V_\Xi \qquad \text{where} \quad V_\Xi = \bigcup_{j = 1}^q V_j.
\]
Applying Algorithm \ref{algorithm:core algorithm} with the base cone generated by $V_0$ is equivalent to solving the instance of LC-LEP defined by $(\poset', \prec_B, \Xi')$ where $\Xi'$ is the restriction of $\rr^m$ to the subset for which the linear functionals defined by each $v \in V_j$ are strictly positive for each $1 \leq j \leq q$. 

In addition to restricting the computation to the polyhedral cones discussed above, we can reuse solutions of smaller PSD problems in some larger computations. As an example, suppose $\poset' = \setof{p'_0, \dotsc, p'_7}$ is the set of interaction polynomials for the PSD with interaction type $\mathbf{n}' = (2,1)$ and $\cP := \setof{p_0,\dotsc, p_{15}}$ the polynomials for the PSD problem with interaction type $\mathbf{n} = (2,1,1)$. Observe that each admissible linear order on $\cP'$ induces an imposed linear order on the even indexed polynomials, $\cP_{\operatorname{even}}:= \setof{p_0,p_{2},\dots, p_{14}} \subset \cP$. A similar linear order is induced on the odd indexed polynomials, $\cP_{\operatorname{odd}} := \setof{p_1,p_{3},\dots, p_{15}} \subset \cP$. 
Hence, a necessary condition to have an admissible linear extension for $\cP$ is that the order of $\cP_{\operatorname{even}}$ and $\cP_{\operatorname{odd}}$ must both be consistent with one of the PSD solutions  in $\totord \left(\cP', \prec_B, (0, \infty)^{6} \right)$. This implies the inclusion 
\begin{equation}
    \label{eq:subproblem}
\totord \left(\poset, \prec_B, (0,\infty)^{8} \right) 
\subseteq \bigcup_{\sigma' \in \totord \left(\cP', \prec_B, (0, \infty)^{6} \right)} \totord(\poset, \prec_B \cup \prec_{\sigma'}, (0,\infty)^8)
\end{equation}
where $\prec_B \cup \prec_{\sigma'}$ represents the refinement of the Boolean lattice partial order, and the partial order induced by $\sigma'$ on the even/odd subsets. 

To exploit this in general, we say that the PSD problem of type $\mathbf{{n}'}$ is a sub-problem for the PSD problem of type $\mathbf{n}$ whenever the polynomials for $\mathbf{n}$ must obey an implied partial order determined by the solutions of $\mathbf{n'}$. Notice that the preceding discussion as well as Equation \eqref{eq:subproblem} applies also to an arbitrary polyhedral cone. Therefore, if there are a total of $k$ admissible linear extensions for all sub-problems of the PSD problem of type $\mathbf{n}$ which we have previously computed, then we bootstrap those results when computing $\totord \left(\poset, \prec_B, (0,\infty)^{2n} \right)$ via the inclusion  
\[
\totord \left(\poset, \prec_B, (0,\infty)^{2n} \right)\subseteq\bigcup_{i=1}^k \totord(\poset,\prec_B\cup \prec_{\sigma_i'}, \Xi') \subseteq  \totord(\poset,\prec_B, \Xi')
\]
where $\prec_B\cup \prec_{\sigma'_i}$ represents the refinement of the Boolean lattice partial order, and the partial order induced by $\sigma'_i$ on the corresponding subsets obtained from any sub-problem. This technique has been used in the computation for all the cases of order $\geq 4$. Observe that the computation of $\totord(\poset,\prec_B\cup \prec_{\sigma_i'}, \Xi')$ can be done distributively for $i=1,\dots, k$ on different computational nodes, which, as is indicated in Section \ref{sec:results}, we employed for the PSD problems of orders $5$ and $6$.

In the special case of Section \ref{sec:special_case}, we proved that inclusion in Equation \eqref{eq:additional_linear_constraints} is actually equality when $\Xi'$ is constructed as we have described. However, in the typical case, these additional algebraic constraints are not sufficient to remove all spurious linear extensions except in the case $\textbf{n} = (2,1,\dots,1)$. 
It remains an open problem to determine a smaller set $\Xi'$ such that $\totord \left(\poset, \prec_B, (0,\infty)^{2n} \right) = \totord \left(\poset', \prec_B, \Xi' \right)$ for other interaction types. 
However, in the remainder of this section we consider the problem of extracting $\totord \left(\poset, \prec_B, (0,\infty)^{2n} \right)$ from $\totord \left(\poset', \prec_B, \Xi' \right)$ when they are not equal.

Observe that we may obtain large subsets of $\totord \left(\poset, \prec_B, (0,\infty)^{2n} \right)$ simply by sampling.
The particular strategy that we adopted is as follows. 
We uniformly sampled between $10^8$ and $10^9$ points 
\[
\xi = (l_1,\dots,l_n,\delta_1,\dots, \delta_n)\in \mathbb{Z}_+^{2n}\cap B^{2n}_{\infty}(r),
\]
where $B^{2n}_{\infty}(r) = \left\{ \|  \xi \|_{\infty} \leq r\right\}$.  
We chose $r=1000$. Mathematically the particular choice of $r$ is not important since the PSD polynomials are homogeneous, though in practice it does have an effect on sampling precision and speed.
For each such $\xi$ we evaluated $\setof{p_\alpha(\xi) : p \in \cP}$.
If $\sigma \in S_{2^n}$ denotes the linear order of these values, then $\xi$ serves as a ``witness'' for the claim that $\Xi_\sigma \neq \emptyset$.
This produces 
\[
\mathcal{S} \left(\poset, \prec_B, (0,\infty)^{2n} \right)  :=  \setof{ \sigma \in \totord \left(\poset, \prec_B, (0,\infty)^{2n} \right) : \sigma \ \text{is witnessed by at least one sample}}.
\]
Obviously,
\[
\mathcal{S}\left(\poset, \prec_B, (0,\infty)^{2n} \right) \subseteq \totord \left(\poset, \prec_B, (0,\infty)^{2n} \right) \subseteq \totord(\poset,\prec_B, \Xi').
\]
In general, sampling is relatively efficient and in cases where $\totord \left(\poset, \prec_B, (0,\infty)^{2n} \right)$ is not too large (see Table \ref{table:results} for details), we recover the entire solution.

Once we have constructed the set $\mathcal{S}\left(\poset, \prec_B, (0,\infty)^{2n} \right), \totord(\poset,\prec_B, \Xi')$ from sampling and algorithms in Section \ref{sec:LCLEP} respectively, we apply a CAD algorithm to check whether or not the semi-algebraic set,
\[
\Xi_\ord = \{\xi\in \Xi : p_{\sigma(0)}(\parm) < p_{\sigma(1)}(\parm) < \dots < p_{\sigma(2^n-1)}(\parm)\},
\]
is empty for each $\sigma \in \totord(\poset,\prec_B, \Xi') \setminus \mathcal{S}\left(\poset, \prec_B, (0,\infty)^{2n} \right)$ and then $\totord \left(\poset, \prec_B, (0,\infty)^{2n} \right)$ is recovered. The CAD algorithm implementation we are using is CylindricalAlgebraicDecomposition in Mathematica 11 \cite{Mathematica}. 

\subsection{Computational complexity}
\label{}

Analyzing the efficiency of our solver for the PSD problem can be broken into three cases.  
The first case is the linear PSD problem.
This is an instance of the LC-LEP in $2n$ variables, and consequently, the complexity analysis in Section \ref{subsection:complexity} applies. 

The second case is the PSD problem with interaction type $\mathbf{n} = (2,1,\dotsc, 1)$. 
By Theorem \ref{thm:characterization_of_equivalent_orders}, the solution set is identical to the solutions of the associated linearized PSD problem. The latter is an instance of LC-LEP in $m =  4 + 2(q-2) = 2(q+1)$ variables and since $q = n-1$, we have $m = 2n$. In other words, the computational cost of solving this PSD problem is again equivalent to solving an instance of the LC-LEP in the same number of variables and the analysis in Section \ref{subsection:complexity} applies. 

The final case is the general PSD problem for interaction type $\mathbf{n} \in \nn^q$, which we assume does not satisfy either of the two previous cases. 
Here we cannot guarantee that the complexity is any better than that of a CAD-based approach since
the final step in our algorithm is to determine the admissibility of every linear extension
\[
\ord \in \mathcal{S}\left(\poset, \prec_B, (0,\infty)^{2n} \right) \setminus \totord(\poset,\prec_B, \Xi').
\]
That is, any total order which is admissible for the linearized PSD problem, but is not witnessed when sampling the nonlinear PSD problem, must be rigorously checked for admissibility. 
This is done using a CAD-based algorithm to certify whether or not the semi-algebraic set, $\Xi_\ord$, is empty for each such $\sigma$. 
As discussed in Section \ref{sec:introduction}, solving this problem requires, again in the worst case, computing a full CAD for $\cP$ subject to the additional algebraic constraints $\ell_i > 0, \delta_i > 0$ for $1 \leq i \leq n$. Consequently, our algorithm for solving the PSD inherits the double exponential complexity of the CAD algorithm in Equation \eqref{eq:CAD_complexity}. Based on the actual computations we have carried out, we conjecture that this worst case bound is not typical and may not even be sharp. There are several heuristics which make this conjecture plausible.  

The first is the fact that sampling alone was sufficient to recover the full solution of the PSD problem in every case examined in this work (compare the first and last columns of Table \ref{table:results}). Consequently, every candidate ordering which was not ruled out by sampling was in fact non-admissible and the CAD-based computation only needed to certify this face. This is advantageous due to the fact that many speedups for the general CAD algorithm apply specifically for proving that a semi-algebraic set is in fact empty. Once again, each of these improvements maintains the same worst case running time, but they often produce much faster typical running times. 

The second heuristic is due to the structure of the PSD problem itself. Not only are the polynomials in $\cP$ linear with respect to each variable, but $\cP$ itself contains a ``complete'' set of polynomials arising from the interaction function which forms a sort of template. This imposes stricter algebraic constraints on the linearized PSD problem as a consequence of Lemma \ref{lem:simple_calculus_ineq}. Additionally, this structure causes the PSD problems to nest into one another neatly in the sense of Equation \eqref{eq:subproblem}. These properties combine to produce sets of candidate solutions of the linearized PSD problem which are remarkably smaller than the candidates produced by solving the linearized PSD problem naively (compare the second and third columns in Table \ref{table:results}). 

These heuristics combined with the results for the cases we have computed make a compelling case that this algorithm is already reasonably efficient, but this is by no means the last word on the subject. 
Further improvements to this approach via imposing additional algebraic constraints, improving sampling techniques, or applying the linearization technique to other classes of polynomials is the subject of ongoing research.

\section{Results for some PSD problems}
\label{sec:results}
In this section we  provide (see Table~\ref{table:results}) the results of our computations for interaction functions of orders 4, 5, and 6. 
A slightly different approach was taken to compute orders 5 and 6, from that used for 4.
This had to do with the machines being used, but highlights the flexibility of our method.

For interaction functions of order 4, we applied Algorithm~\ref{algorithm:incone} using a rational linear programming algorithm. 
In particular, we used the  implementation  MixedIntegerLinearProgram from SageMath 8 \cite{sagemath}.
This implies that the output of Algorithm~\ref{algorithm:core algorithm}  is correct.
Observe that interaction type $(4)$ is linear and type $(1,1,1,1)$ is log linear, and therefore Algorithm~\ref{algorithm:core algorithm} produces $\totord\psdR$.
The fact that our output agrees with that of \cite{boolorder} suggests that our code is functioning as desired.
To compute the interaction type $(2,1,1)$ we apply Algorithm~\ref{algorithm:core algorithm} to obtain $\totord\lpsd$.
By Theorem~\ref{thm:characterization_of_equivalent_orders} this determines $\totord\psdR$.

To solve the PSD problem from interaction types $(2,2)$ and $(3,1)$ requires that we make use of the strategy discussed in Section~\ref{sec:SolvingPSDGeneral}.
Again, we use Algorithm~\ref{algorithm:core algorithm} to obtain $\totord\lpsd$.
By Theorem~\ref{thm:P'_is_boolean}, $\totord\psdR \subset \totord\lpsd$.
As indicated in Column 7 of Table~\ref{table:results}, we chose $10^8$ samples from $(0,\infty)^8$ and identified $5344$ and $3084$ linear orders, respectively.
We ran CylindricalAlgebraicDecomposition in Mathematica 11 \cite{Mathematica} on each element of $\totord\lpsd \setminus \mathcal{S}\left(\poset, \prec_B, (0,\infty)^{2n} \right)$.
As can be seen by comparing Columns 6 and 3, none of these elements were admissible.

\begin{table}[h]
	\begin{tabular}{|c |c | c|  c| c |    }
		\hline
 $\mathbf{n}$ & $ \# \totord\psdR $ & $ \# \totord(\poset',\prec_B, \Xi')$ & $ \# \totord\lpsd $  &  $ \# \mathcal{S}(\poset',\prec_B, (0,\infty)^{2n})$    \\ \hline
(1,1,1,1) & 336 & -&-  &-\\ \hline
(4) & 336 & - & - & - \\ \hline
(2,1,1) & 1,344 & 1,344 &2,352 &- \\ \hline
(2,2) &  5,344  & 7,920 & 26,640 & 5,344\\ \hline
(3,1) &  3,084  &  5,112 &68,641 & 3,084\\ \hline\hline
(1,1,1,1,1) & 61,920 & - & - & 61,920 \\ \hline
(5) & 61,920 & - & - & 61,920  \\ \hline
(2,1,1,1) & 790,200  &  790,200 & * & 790,200 \\ \hline
(2,2,1) &  - & 11,035,808 & * & 6,570,952  \\ \hline
(3,2) & - & * &*&  71,959,088$^{\dagger}$  \\ \hline
(4,1) & - & * & *&  11,213,616$^{\dagger}$ \\ \hline\hline
(1,1,1,1,1,1) & 89,414,640 & - & - & 89,414,640 \\ \hline
(6) & 89,414,640 & - & - & 89,414,640  \\ \hline
	\end{tabular}
	\caption{Computational results for several PSD problems. 
	Column 1 indicates the interaction type.
	Column 2 provides the number of elements in the AC-LEP of interest.
	Column 3 provides the number of elements in an associated LC-LEP.  This is not relevant where the AC-LEP problem of interest is a LC-LEP problem and is indicated by -. The * indicates that the computation was too large to complete.
	Column 4 provides the number of elements in the linearized PSD problem without additional constraints. Again the irrelevance for linear problems is indicated by - and * indicates that the computation is large to be performed.
	The last column indicates the number of cells identified via sampling. We used $10^8$ samples for all $n=4$ cases and $10^9$ samples for the $n = 5,6$ cases. The symbol $^{\dagger}$ indicates that our sampling was not sufficient.
	}
	\label{table:results}
\end{table}

We now turn to the computations of interaction functions of order 5 and 6.
As these problems are too big to be done on a laptop we turned to a server for which SageMath was not installed.
Thus, we made use of a numerical linear programing algorithm,  linprog from Python 3.5 package scipy \cite{2020SciPy-NMeth} with the default numerical error $10^{-13}$, in Algorithm \ref{algorithm:incone}.
The interaction type type $(5)$ and $(6)$ are linear and $(1,1,1,1,1)$ and $(1,1,1,1,1,1)$ are log linear, and therefore via Algorithm~\ref{algorithm:core algorithm} we obtain $\totord_{alg}\psdR$.
We use the sampling technique (see Columns 7 and 8 of Table~\ref{table:results}) to verify each of the elements of $\totord_{alg}\psdR$, thereby obtaining $\totord\psdR$.

The computation for each  order 4 case was done on a Mac Pro laptop (2.7 Hz Intel i5 and memory 8GB) with computation time under 4 hours.  The computation of the remaining cases were done using a computing server with CentOs, intel 17.1, memory 32 GB, and less than 30 nodes. The computation time for both $(1,1,1,1,1)$ and $(5)$ was less than 4 hours, while the computation time for $(2,1,1,1)$ was on the order of 7 days. The codes which produced all of the computations in Table \ref{table:results} are available on \href{https://github.com/lunzhang1990/parameterRepo}{GitHub}.

\subsection*{Acknowledgments}
The authors would like to thank Shaun Harker, Sandra Di Rocco, Tomas Gedeon, and Mike Saks for helpful conversations. 
The authors acknowledge support from  NSF DMS-1521771, DMS-1622401, DMS-1839294, the NSF HDR TRIPODS award CCF-1934924, DARPA contracts HR0011-16-2-0033 and FA8750-17-C- 0054, and NIH grant R01 GM126555-01. 
\bibliographystyle{plain}
\bibliography{library}

\end{document}